\documentclass[11pt,reqno]{amsart}

\newcommand{\la}{\langle}
\newcommand{\ra}{\rangle}

\renewcommand{\Im}{\operatorname{Im}}

\newcommand{\defeq}{\stackrel{\rm{def}}{=}}

\usepackage{amssymb}
\usepackage{amsthm}
\usepackage{amsxtra}
\usepackage{mathabx}
\usepackage{fancyhdr}
\usepackage{graphicx}
\usepackage{color}
\usepackage{amsmath}
\usepackage{listings} 
\usepackage{float}
\usepackage{cases} 
\usepackage{threeparttable}
\usepackage{dcolumn}
\usepackage{multirow}
\usepackage{booktabs}
\usepackage{comment}

\usepackage[normalem]{ulem}

\newcommand{\R}{\mathbb R}

\newcommand\norm[1]{\left\lVert#1\right\rVert}
\newcommand{\vertiii}[1]{{\left\vert\kern-0.25ex\left\vert\kern-0.25ex\left\vert #1 
    \right\vert\kern-0.25ex\right\vert\kern-0.25ex\right\vert}}

\newtheorem{theorem}{Theorem}[section]
\newtheorem{definition}[theorem]{Definition}
\newtheorem{proposition}[theorem]{Proposition}
\newtheorem{lemma}[theorem]{Lemma}
\newtheorem{corollary}[theorem]{Corollary}

\numberwithin{equation}{section}

\theoremstyle{definition}
\newtheorem{remark}[theorem]{Remark}
\newtheorem{example}[theorem]{Example}

\author[A. D. Rodriguez]{Alex D. Rodriguez}
\address{Department of Mathematics  \& Statistics\\Florida International University,  Miami, FL, USA}
\curraddr{}
\email{arodr1128@fiu.edu}

\author[G. Azcoitia]{Gia Azcoitia}
\address{Department of Electrical and Computer Engineering, UCLA, CA, USA}

\curraddr{}
\email{gia624@g.ucla.edu} 

\author[H. Wubben]{Hannah Wubben}
\address{Department of Mathematics  \& Statistics\\Northern Arizona University, Flagstaff, AZ, USA}
\curraddr{}
\email{10hannahwubben@gmail.com}

\author[S. Roudenko]{Svetlana Roudenko}
\address{Department of Mathematics  \& Statistics\\Florida International University,  Miami, FL, USA}
\curraddr{}
\email{sroudenk@fiu.edu}

\title[Well-posedness Methods for 1D NLS]{Review of well-posedness methods\\ for the 1D nonlinear Schr\"odinger equation\\ with an application to combined nonlinearities}



\begin{document}

\begin{abstract} We consider the nonlinear Schr\"{o}dinger equation in one dimension with nonlinearities of type $|u|^\alpha u$ for any power $\alpha>0$ and review two different methods for obtaining solutions, namely, local well-posedness, with initial data either in $L^2$ or $H^1$, or in the weighted subspace of $H^1$. One approach is based on the Strichartz estimates, and thus, $H^1$ well-posedness typically holds for nonlinearities with power $\alpha \geq 1$. The other one is a direct application of weighted estimates commuting with derivatives and a certain infimum condition on the initial data, and thus, can treat nonlinearities for the whole range $0 < \alpha < \infty$; furthermore, it can handle a sum of different nonlinearities. We then conclude with an application of the second approach to the NLS with {\it finitely} many combined nonlinearities, important for physical applications (e.g., in laser optics), as it is more challenging, if at all possible, to obtain local well-posedness with the first method due to the lack of scaling invariance.   
\end{abstract}

\maketitle

\section{Introduction} \label{S: Intro}
We consider the one dimensional nonlinear Schr\"{o}dinger (NLS) equation
\begin{equation}\label{NLS}
 i \partial_t u + \partial_{x}^{2}u + \mathcal N(u) 
         = 0, \quad x \in \mathbb R, ~t \in \mathbb R, 
\end{equation}
with the initial data $u(x,0) = u_{0}$ taken in some Sobolev-type space. This initial value problem, often called the Cauchy problem for the NLS equation, can have different approaches in order to obtain solutions depending on the nonlinearity $\mathcal N(u)$;  
in this paper we review two different methods that produce solutions to the above Cauchy problem with a single nonlinearity 
\begin{equation}\label{E:one}
\mathcal N(u) = \lambda |u|^\alpha u, ~ \alpha >0, ~\lambda \in \mathbb R, 
\end{equation}
and then apply the second method to the NLS equation with {\it combined} nonlinearities (cNLS),  
\begin{equation}\label{cNLS}
\qquad \mathcal N(u) = \lambda_{1}|u|^{\alpha_1}u + \lambda_{2}|u|^{\alpha_2}u ... + \lambda_{N}|u|^{\alpha_N}u, \quad \alpha_i > 0, ~\lambda_i \in \mathbb C, i=1, ..., N, ~N \in \mathbb N,
\end{equation} 
to show that this more recent approach, introduced by Cazenave \& Naumkin in \cite{CN2017}, can be applied when some standard symmetries (such as scaling in this case) are broken.\footnote{For simplicity and conciseness of this paper we consider the problem in $1$D, however, both methods work in higher dimensions as well, though with certain restrictions on nonlinearity power $\alpha$.} The first approach is based on so-called Strichartz estimates, the classical method to obtain the local well-posedness of solutions (i.e., existence, uniqueness and continuous dependence on the initial data), however, this approach does not cover the entire range of nonlinearities (it often has some restrictions on the power  $\alpha$ such as $\alpha \geq1$) and is problematic to apply for combined nonlinearities such as in \eqref{cNLS} or different type of nonlinearities (such as exponential); the second approach can address the full range of nonlinearities, and furthermore, different types of nonlinearities, however, it has other restrictions such as non-changing sign of solutions or a polynomial decay of initial data. Both methods are based on a fixed point theorem via a contraction mapping argument, however, how to obtain the control of linear and nonlinear parts are different.

In the case of one nonlinear term \eqref{E:one}, the NLS equation is well-studied in the literature with many local and global results available not only in $H^1$ but also in $H^s$, $s<1$, and also in higher dimensions (for example, see \cite{GV1979}, \cite{GV1985}, \cite{Kato1987}, \cite{T1987}, or books \cite{Caz-book}, \cite{LPIntroToNDE2014}, \cite{tao2006nonlinear}). Historically, these results have been obtained using space-time estimates, colloquially called Strichartz estimates. This approach has challenges, which typically place restrictions on the power $\alpha$ (and dimension $N$, if studied in higher dimensions). Recently, there have been advances in developing methods for local and global theory to open up the aforementioned restrictions, for example, by Cazenave and Naumkin \cite{CN2017}, who suggested using weighted estimates to eliminate previous restrictions of power and dimension, while introducing a ``non-vanishing" or ``infimum" condition, where the initial datum is required to be bounded away from zero, see \eqref{Xinf}, to handle derivatives for powers in the nonlinear term $0 < \alpha < 1$. {For example, for $\alpha < 1$, a derivative of the nonlinear term yields $|\partial_x(|u|^\alpha u)| \sim \alpha|u|^{\alpha-1}|\partial_xu|$ , which introduces a singularity at points where $u(x) = 0$.} 

To further show the power of the second method, we apply it to the NLS with combined nonlinearities \eqref{cNLS}. In fact, we prove the local well-posedness for {\it finitely many} combined terms, then, taking all but one $\lambda_i = 0$, we deduce it for one term. By {\it well-posedness} it is typically meant a combination of {\it existence} of solutions, their {\it uniqueness} and {\it continuous dependence} on initial data. By {\it local} well-posedness it is meant that it is valid for some time $0< T <\infty$ and by {\it global} if it holds for any $T>0$. It should be mentioned that well-posedness (as well as the blow-up criteria) for the cNLS equation with initial data in $H^1(\mathbb{R}^N)$ and two nonlinear terms in dimension $N\geq 1$ has been studied, for example, in \cite{TVZ2005}, \cite{Z2006}.  
While those results cover most of the 1D case we consider here, when it comes to the modification of the nonlinearities and, for instance, handling three or finitely many terms, that approach becomes incredibly tricky to handle, if at all possible. 
If the nonlinearities are modified to a different type such as convolution (as in the  Hartree equation) or considering an NLS equation on a torus with log-type or even inverse polynomial type nonlinearities, then the first method becomes inapplicable, while the second one works, see, for instance, \cite{ARR2021}, \cite{RRS2025}.

The physical applications of the standard NLS are vast and well studied, and the NLS with combined nonlinear terms is no different. Indeed, these types of nonlinearities are applicable, for example, to the propagation of plasma waves and laser beams, the solitary wave propagation as signals in fiber optics as waves in fluid mechanics, see the classical textbooks of \cite{SulemSulem1999}, \cite{F2015}, as well as Bose-Einstein condensate \cite{Kengne2008}.   
It was long understood that a more appropriate physical realization would require perturbations of one nonlinear term such as the cubic nonlinearity by higher order powers \cite{ZakRub1973}, \cite[\S 3]{Z1971}, \cite{BGMP1988}, \cite{PAV1996}, more recent works address two combined nonlinearities  
such as cubic-quintic \cite{Pushkarov1979}, \cite{Malkin1993}, \cite{BG2001},\cite{UST}, \cite{M2019}, cubic-quartic \cite{AS2023}, 
and quadratic-cubic \cite{M2019}, triple combined nonlinearity appears such as cubic-quintic-septic in \cite{RMA2015} or quadratic-cubic-quartic in \cite{LTZ2021}, and quadruple nonlinearity such as cubic-quintic-septic-nonic in \cite{4nonlin}. 
\smallskip

{  
We now state the theorems proved in this paper, starting with the first approach where one uses the classical Strichartz estimates, obtaining the local well-posedness in $L^2$ and then in $H^1$ to show how to increase the smoothness or regularity of solutions. We present Theorems \ref{T:1} and \ref{T:2} utilizing Strichartz estimates in an expository manner; for further references on that refer to \cite[Section 2.3]{Caz-book} or \cite[Chapter 5]{LPIntroToNDE2014}. Theorems \ref{T:3} and \ref{T:4} are proved with a more recent method developed in \cite{CN2017}. We note that Theorem \ref{T:4} is a novel application and extension of Theorem \ref{T:3} (proved in \cite{CN2017}, see further generalization in \cite{RRR2026}); for notation and definitions, refer to Section \ref{S:notation}.}

\begin{theorem}[Local well-posedness in $L^2(\mathbb{R})$]\label{T:1}
Let $0 < \alpha < 4$ and $u_0 \in L^2(\mathbb{R})$. Then there exist a time $T = T(\|u_0\|_{L^2(\mathbb{R})}, \lambda, \alpha) > 0$ and a unique solution $u(x,t)$ of the Cauchy problem \eqref{NLS}-\eqref{E:one} with $u(x,0)=u_0$ in the time interval $[-T,T]$ such that
$$
u \in C([-T,T],L^2(\mathbb{R})) \cap L^q([-T,T],L^r(\R)),
$$
where the pair $(q,r) = (\frac{4(\alpha+1)}{\alpha}, 2(\alpha+1) )$.

Moreover, for any $0<\widetilde{T}<T$, there exists a neighborhood $V$ in $L^2(\mathbb{R}) \cap L^{r}(\R)$ such that the map $u_0 \mapsto u(\cdot,t)$ is Lipschitz continuous from $V$ into the space $C([-T,\widetilde{T}],L^2(\mathbb{R})) \cap L^q([-T,\widetilde{T}],L^r(\R))$. 
\end{theorem}

Since in Theorem \ref{T:1} there is a restriction on nonlinearity power $\alpha$, we upgrade the regularity of solutions to $H^1$ and show the well-posedness in this space, covering in 1D  nonlinearities $1 \leq \alpha < \infty$.  
\begin{theorem}[Local well-posedness in $H^1(\mathbb{R})$]\label{T:2}
Let  $1 \leq \alpha <\infty$ and $u_0 \in H^1(\mathbb{R})$. Then there exist $ T = T(\|u_0\|_{H^1(\mathbb{R})}, \lambda, \alpha) > 0$ and a unique solution $u(x,t)$ of the Cauchy problem \eqref{NLS}-\eqref{E:one} with $u(x,0) = u_0$ in the time interval $[-T,T]$ such that
$$
u \in C([-T,T],H^1(\mathbb{R})) \cap L^q([-T,T],W^{1,r}(\R)),
$$
where the admissible pair $(q,r)$ satisfies 
\eqref{E:rq}.

Moreover, for any $0<\widetilde{T}<T$, there exists a neighborhood $V$ in $H^1(\mathbb{R}) \cap W^{1,r}(\R)$ such that the map $u_0 \mapsto u(\cdot,t)$ is Lipschitz continuous from $V$ into the class $C([-\widetilde{T},\widetilde{T}],H^1(\mathbb{R})) \cap L^q([-\widetilde{T},\widetilde{T}],W^{1,r}(\R))$.
\end{theorem}

The next theorem uses more basic tools (such as the Fundamental Theorem of Calculus and Taylor expansion), though involves some amount of careful computations (of various orders derivatives and dealing with polynomial weights via the infimum condition). It gives local well-posedness in a subspace of $H^1$, which is denoted just by the space $\mathcal X$. 
We start with the definition of this space.
\smallskip

Take a real number $n$ and a positive integer $r$ such that 
\begin{equation}\label{XspaceCond1}
      {r  \geq  3}, \; 0 < n \leq r,
\end{equation}
and choose $M \in \mathbb{N}$ 
such that
\begin{equation}\label{XspaceCond2}
    n+r\leq M \leq 2r.
\end{equation}
Denote $\langle x \rangle^{n} \defeq (1 + |x|^2)^{\frac{n}2}$ and for the above parameters $n,r,M$, define the space $\mathcal X$ as
\begin{equation}{\label{Xspace}}
    \mathcal{X} = \Bigg\{ \partial_xu, \partial_x^2u, ...,\partial_x^Mu \in L^2(\mathbb{R}) :
    \left( {\footnotesize
    \begin{array}{cl}
        {\rm I.} & \langle{x}\rangle^{n}u\in{L^{\infty}}(\mathbb{R})\\
        {\rm II.} & \langle{x}\rangle^{n}\partial_{x}^\beta{u}\in{L^2}(\mathbb{R}) \; ~\text{for}~ \; 1 \leq \beta \leq r\\
        {\rm III.} & ~~\partial_{x}^\beta{u} \in L^2(\mathbb{R}) \; ~\text{for}~ \; r+1 \leq \beta \leq M
    \end{array}
    }
    \right) ~
    \Bigg\},
\end{equation}
equipped with the norm
\begin{equation}{\label{Xnorm}}
    \norm{u}_{\mathcal{X}} \defeq \norm{\langle x \rangle^n u(x)}_{L^{\infty}} + \sum_{\beta = 1}^{r} \norm{\langle x \rangle^{n} \partial_{x}^{\beta} u}_{L^2} + \sum_{\beta = r+1}^{M}\|\partial_x^\beta u\|_{L^2} < \infty. 
\end{equation}

From the definition of our space \eqref{Xspace}, the function $u(x)$ does not need to belong to $L^2(\mathbb R)$, but all its derivatives up to order $M$ do. (If $n > \frac{1}{2}$, then $u(x) \in L^2(\mathbb{R})$, verifying this is a short exercise for an interested reader).
\begin{remark}\label{R:1}
An example of a function in the space $\mathcal X$ is $u(x) = \frac{A}{\la x \ra^n}$ for any non-zero constant $A$.  
\end{remark}
Since we would like to consider small $\alpha$ in nonlinearity $|u|^\alpha u$, to be able to differentiate it (and thus, handle the term $|u|^{\alpha-1}$, see Remark \ref{R:3}), we introduce the {\it non-vanishing} condition 
\begin{equation}{\label{Xinf}}
        \inf_{x\in\mathbb{R}} ~\langle x \rangle^n |u(x)| > 0.
    \end{equation}
An example of a function from Remark \ref{R:1} satisfies this condition. Since the local existence time will depend on the infimum value in \eqref{Xinf}, we rewrite this condition specifying a lower bound as
\begin{equation}{\label{inf}}
     \eta\inf_{x\in\mathbb{R}}(\langle x \rangle^n |u(x)|) \geq 1 \quad \mbox{for ~~ some }~~ \eta >0.
\end{equation}
\smallskip

We are now ready to state the well-posedness result in a weighted subspace of a Sobolev space. 
\begin{theorem}[Local well-posedness in space $\mathcal{X}$]\label{T:3}
Let $u_0 \in \mathcal{X}$ satisfy \eqref{Xinf}. 
Then there exists a time $T = T(\lambda,\alpha, n, \|u_0\|_{\mathcal X}, \eta)>0$ and a unique solution $u(x,t)$ of the Cauchy problem \eqref{NLS}-\eqref{E:one} with $u(x,0) = u_0$ in the time interval $[-T,T]$ such that
$u \in C([-T,T], \mathcal{X})$.

Moreover, for any $0<\widetilde{T}<T$, there exists a neighborhood $V$ of $u_0$ in $\mathcal{X}$ satisfying \eqref{Xinf} such that the map $u_0 \mapsto u(\cdot,t)$ is Lipschitz continuous from $V$ into the class $C([-\widetilde{T},\widetilde{T}],\mathcal{X})$. 
\end{theorem}

To compare with Theorem \ref{T:2}, where we obtained solutions in the space $H^1(\mathbb{R})$ (intersected with $L^q_t W^{1,r}_x$), in this theorem we study solutions to \eqref{NLS}-\eqref{E:one} in the space $\mathcal{X}$, subset of $H^1(\mathbb{R})$. While the class of functions in the second case is smaller, it does not use such involved machinery as Strichartz estimates, as the first method, but instead relies on just basic calculus tools, which are accessible to undergraduates. 

To highlight the flexibility of the second method, we show its application to the cNLS equation with finitely many nonlinear terms as in \eqref{cNLS}. In fact, in order to prove Theorem \ref{T:3}, we deduce it as a particular case of the next theorem. 

Given nonlinearity powers $\alpha_1$, $\alpha_2$, ..., $\alpha_N$ as in \eqref{cNLS}, we take $n > 0$ (as before)
and choose $r$ and $M$ as in \eqref{XspaceCond1} and \eqref{XspaceCond2}. With these parameters, consider the space $\mathcal X$ as in \eqref{Xspace}-\eqref{Xnorm}.
\begin{theorem}[Local well-posedness for cNLS in $\mathcal{X}$]\label{T:4}
Let $\alpha_1,\alpha_2, ..., \alpha_N > 0$ and $\mathcal{X}$ be defined by \eqref{Xspace} and \eqref{Xnorm}. 
Let $u_0 \in \mathcal{X}$ and satisfy \eqref{Xinf}. Then there exist a time $T = T(\lambda_1, ..., \lambda_N, \alpha_1, ..., \alpha_N, n,$ $ \|u_0\|_{\mathcal X}, \eta)>0$ and a unique solution $u(x,t)$ of the Cauchy problem \eqref{NLS} \& \eqref{cNLS} with $u(x,0) = u_0$ in the time interval $[-T,T]$ such that $u \in C([-T,T], \mathcal{X}).$

Moreover, for any $0<\widetilde{T}<T$, there exists a neighborhood $V$ of $u_0$ in $\mathcal{X}$ satisfying \eqref{Xinf} such that the map $u_0 \mapsto u(\cdot,t)$ is Lipschitz continuous from $V$ into the class $C([-\widetilde{T},\widetilde{T}],\mathcal{X})$.
\end{theorem}
We note that the condition of {\it finitely many} nonlinear terms is essential, since the time of existence $T$ in Theorem \ref{T:4} inversely depends on the sum of finitely many terms of specific non-zero size quantities (see \eqref{E:T_weight} and the comment before that). It would be interesting to investigate if this approach can be modified to accommodate infinitely many terms. 
 
\smallskip

This paper is organized as follows: Section \ref{S: Prelim} contains preliminary information, including useful inequalities that are used throughout the paper. In Section \ref{S:Strichartz} we walk the reader through the classical method to prove local well-posedness, using Strichartz estimates for the equation \eqref{NLS}-\eqref{E:one}, pointing out restrictions and challenges either in $L^2$ or $H^1$ settings. In Section \ref{S: Weights} we describe the second approach: Section \ref{S: Weights}.1 contains linear estimates (proof of Theorem \ref{linearEst}), Section \ref{S: Weights}.2 contains the nonlinear estimates (proof of Theorem \ref{NonlinearEst}), and finally, Section \ref{S: Weights}.3 contains the proof of Theorem \ref{T:4} and, as a special case, of Theorem \ref{T:3} with a remark afterwards.
\smallskip

{\bf Acknowledgments.}
This project was initiated during the Summer 2021 REU program ``AMRPU@FIU" (Applied Math Research Program for Undergraduates at Florida International University); 
G.A. and H.W. were partially supported by the NSF REU grant DMS-2050971 and NSA grant H982302210016 (PI: S. Roudenko). A.D.R. and S.R. were partially supported by NSF grants DMS-1927258, 2055130, 2221491, 2452782.
The authors would like to thank Oscar Ria\~no for numerous insightful conversations related to this work. The authors are grateful to the anonymous reviewers for their careful reading of the manuscript and for providing comments and suggestions for its improvement.

\section{Preliminaries}\label{S: Prelim}
This article makes use of various definitions, theorems and other concepts typically seen in upper undergraduate or beginning graduate level courses. Indeed, the authors G.A. and H.W. both began this project with only a strong background in Calculus. 
All theorems in this paper use a fixed-point argument, based on a Banach fixed-point theorem (e.g., \cite[Chapter 5]{K1979Book}), by showing that a certain map is a contraction (e.g., \cite[Theorem 6.9, pg 361]{SSbook2011}), and thus, has a unique point, which will be a solution in a space, where contraction was proved. 
To obtain the estimates for the contraction, this is where the methods differ, as well as the function space on which the mapping is considered. The estimates in each space proceed with obtaining linear estimates first and then the nonlinear ones. 
Before starting with a short review of the 
basic theory of solutions for the linear Schr\"odinger equation and other useful inequalities, we give notations and definitions for the standard objects used in this paper, such as function spaces and operators used.  

\subsection{Notation.}\label{S:notation}
We use the standard Lebesgue spaces $L^p(\mathbb{R})$, $1\leq p\leq \infty$, with the norm
$$
\|f\|_{L^p} = \left\{ \begin{array}{ll} 
\left( \int_{-\infty}^\infty |f(x)|^p \, dx \right)^{1/p} & \text{if } p < \infty, \\
\text{ess} \sup_{x \in \mathbb{R}} |f(x)| & \text{if } p = \infty.
\end{array} \right.
$$

The notation 
$L^p_x$ is used to specify that the $L^p$-norm is acting only on the $x$-variable. 

The Fourier transform of a sufficiently smooth function (for example, function of moderate decrease as in \cite[p.131]{SSbook2011Fourier} or Schwartz class $\mathcal S$ function \cite[p.134]{SSbook2011Fourier}) $f(x)$ is defined as  
$$
\widehat{f}(\xi) = 
\int_{-\infty}^\infty f(x)e^{-2\pi i x \xi} \, dx.
$$
The inverse Fourier transform is defined in a similar manner, ${f}^\vee (x) = 
\int_{-\infty}^\infty f(\xi)e^{2\pi ix \xi} \, d\xi$ (sometimes in literature also denoted by $\mathcal F^{-1}(f)$). One important property of the Fourier transform is the Plancherel formula (the isometry on $L^2$) 
$$
\|f\|_{L^2} = \|\hat{f}\|_{L^2}, \qquad f \in \mathcal S(\mathbb R). 
$$

Another fundamental property of the Fourier transform is $\widehat{f'(x)} = 2\pi i \xi \hat{f}(x)$, which says that 
differentiation in $x$ is equivalent to simple multiplication by $\xi$ in the frequency space (Fourier space). Consequently, many PDEs become significantly easier to solve once they are moved `to the other side', and we demonstrate that in the next subsection on an example of the linear Schr\"odinger equation. 
For other properties and details on the Fourier transform, see \cite[Chapter 5]{SSbook2011Fourier}.

For $s\in \mathbb{R}$, the space $H^{s}(\mathbb{R})$ denotes the $L^2$-based Sobolev space of order $s$ with the norm 
$$
\|f\|_{H^s_x} = \| \langle \xi \rangle^s \hat{f}(\xi) \|_{L^2_\xi} = \Big( \int_{-\infty}^\infty (1 + |\xi|^2)^s |\hat{f}(\xi)|^2 \, d\xi \Big)^{1/2}.
$$

If $k \in \mathbb N$, then $\|f\|_{H^k}$ is equivalent to $\|f\|_{L^2}+\|\partial_x^k f\|_{L^2}$, where $\partial_x$ denotes a derivative with respect to $x$, $\partial_x^2 = \partial_x (\partial_x)$, etc. (for higher derivatives a product-type rule will be used as discussed in Remark \ref{R:2} and Example \ref{Ex:1}). The reason the subscript $x$ is written and the notation of a partial derivative is used, though we are in one dimension, is to be consistent with the literature, e.g., \cite{CN2017}, and motivate future interest in extending to higher dimensions. 

The space $W^{k,p}(\mathbb R)$ denotes a more general Sobolev space with up to $k$ derivatives in $L^p$ space, for this work it is sufficient to consider 
$$W^{k,p}(\mathbb R) = \big\{ f \in L^p(\mathbb R) : \partial^\beta_x f \in L^p(\mathbb R) \text{ for all } |\beta| \leq k \big\}.
$$ 
For further details, for example, see \cite[Chapter 1]{Caz-book}. 

The Bessel potential of order $-s$ is denoted by $J^s=(1-\partial_x^2)^{\frac{s}{2}}$, equivalently (recalling the above property of derivatives under the Fourier transform), $J^s$ is defined (up to a constant) by the Fourier multiplier with symbol $\langle \xi \rangle^{s}=(1+|\xi|^2)^{\frac{s}{2}}$, thus, $(J^sf)(x)$ should be understood as $\big( \langle \xi \rangle^s \hat{f}(\xi) \big)^{\vee} (x)$. In the light of the Sobolev spaces introduces above, $\| f\|_{H^s_x}  = \|J^{s} f\|_{L^2_x}$.

We are now ready to discuss the Schr\"odinger equation in its linear form and its solutions.

\subsection{Linear Schr\"odinger equation} 
Consider a {\it linear} homogeneous initial value problem:
\begin{equation}{\label{LinearSE}}
    \begin{cases}
        i \partial_t u = -\partial_{x}^{2}u,  \quad t \in \mathbb{R}, ~~ x \in \mathbb{R},\\
        u(x,0) = u_{0}.\\
    \end{cases}
\end{equation}
The space for the initial value function $u_0(x)$ will be specified later, for now, we assume that it is a smooth, sufficiently rapidly decaying at infinity function. Taking the Fourier transform with respect to the space variable $x$, we obtain
\begin{equation*}
    \begin{cases}
        \hat{u}_{t}(\xi,t) = -4\pi^2i\xi^{2}\hat{u}(\xi,t), \\
        \hat{u}(\xi,0) = \hat{u}_{0}(\xi). \\
    \end{cases}
\end{equation*}
Notice that for each fixed $\xi$, the PDE \eqref{LinearSE} has become an ODE in the variable $t$, and thus, we can solve it with separation of variables, to obtain 
\begin{equation}\label{E:Fm}
    \hat{u}(\xi,t) = e^{-4\pi^2it\xi^2}\hat{u}_0(\xi).
\end{equation}
Inverting the Fourier transform (at least formally for now, assuming that the right side is well-defined), we obtain
\begin{align}\label{E:Operator}
u(x,t) = \left(e^{-4\pi^2it\xi^2}\hat{u}_0(\xi)\right)^{\vee} 
    = \left(e^{-4\pi^2it\xi^2}\right)^{\vee}* u_0(x) 
    \defeq  e^{it\partial_x^2}u_0(x). 
\end{align}

The notation $e^{it\partial_x^2}$ is referred to as the linear {\it Schr\"odinger operator}, or the linear Schr\"odinger flow as it is a continuous function in $t$, which follows from \eqref{E:Operator} and properties of the integral. It is also called the Schr\"odinger group as it exhibits group properties, e.g., see \cite[Section 4.1]{LPIntroToNDE2014}.  
The expression $e^{-4 \pi^2 \, it \, \xi ^2}$ is called
the Fourier multiplier associated to the Schr\"odinger group: on the Fourier side that is exactly what it does as can be seen in \eqref{E:Fm}. 
We next describe one of the most powerful properties of the Schr\"odinger group, the isometry, which can be thought of square-integrating \eqref{E:Fm} (for further details we refer the reader to \cite[Proposition 4.2]{LPIntroToNDE2014}).
\begin{proposition}[Isometry on $L^2$]\label{L2Isometry}
For all $t \in \mathbb{R}$, the linear Schr\"odinger operator $e^{it\partial_x^2} : L^2(\mathbb{R}) \to L^2(\mathbb{R})$ is an isometry with
\begin{equation}\label{E:isom}
\|e^{it\partial_x^2}f \|_{L^2(\mathbb{R})} = 
\|f \|_{L^2(\mathbb{R})}.
\end{equation}
\end{proposition}

{ 
We are now ready to discuss the solutions to the {\it nonlinear} Schr\"odinger equation. This is where there can be different approaches on how to control the estimates and in which space the solution may live. Since in this paper we show two approaches, one more classical and another one, a more recent, we describe below preliminary lemmas and propositions needed later. For the rest of this section, we additionally mark them to indicate in which approach they are used. In particular, items useful for

- the first approach, method I via Strichartz estimates, is marked as (\textbf{M.I}), and

- the second approach, method II via weighted Sobolev estimates, is marked as (\textbf{M.II}).

If not marked, they are useful for both methods.}

\subsection{Strichartz estimates}
The following estimates, for which \eqref{E:isom} can be viewed a special case, are at core of the first method to obtain solutions to \eqref{NLS}, proof of which and further details can be found, for example, in \cite[Section 4.1]{LPIntroToNDE2014}. 
\begin{definition}[\textbf{M.I}]
The pair $(q,r)$  is called admissible if 
\begin{equation}\label{admissible}
     \frac{2}{q} + \frac{1}{r} = \frac{1}{2} \quad \text{ and } \quad 2 \leq r \leq \infty .
\end{equation}
\end{definition}

We also use the conjugate notation $q'$ (and similarly for $r'$) if
$$
\frac{1}{q} + \frac{1}{q^\prime} = 1.
$$

\begin{theorem}[\textbf{M.I}] \label{T:StrichartzEstimates} 
{\rm (Strichartz estimates)}
The group $\{e^{it\partial_x^2}\}$ 
satisfies
\begin{equation}\label{groupprop1}
    \left(\int_{-\infty}^\infty\norm{ e^{it\partial_x^2}f} _{L^r(\mathbb{R})}^qdt\right)^\frac{1}{q} \leq C\|f\|_{L^2(\mathbb{R})}
\end{equation}
and
\begin{equation}\label{groupprop4}
    \left(\int_{-\infty}^\infty\norm{\int_{0}^t e^{i(t-s)\partial_x^2}f ds}_{L^{r}(\mathbb{R})}^qdt\right)^\frac{1}{q} \leq C \Big( \int_{-\infty}^\infty\norm{ f}_{L^{r^\prime}(\mathbb{R})}^{q^\prime}dt\Big)^\frac{1}{q^\prime},
\end{equation}
provided $(q,r)$ is an admissible pair.
\end{theorem}

\begin{corollary}[\textbf{M.I}]\label{admisspairestimate}
    Let the pairs $(q, r)$ and $(\tilde q, \tilde r)$  
    satisfy \eqref{admissible}. Then for all $T>0$, we have
    \begin{equation}\label{E:Str}
    \left(\int_{0}^T\norm{\int_{0}^t e^{i(t-s)\partial_x^2}f ds}_{L^{r}(\mathbb{R})}^{q} dt\right)^\frac{1}{q} \leq C \Big(\int_{0}^T\norm{ f}_{L^{{\tilde r}^\prime}(\mathbb{R})}^{{\tilde q}^\prime}dt\Big)^\frac{1}{{\tilde q}^\prime}.
\end{equation}
\end{corollary}
Note that the admissible pairs $(q,r)$ and $(\tilde q, \tilde r)$ in \eqref{E:Str} could be chosen independently of each other. 

\subsection{Conserved quantities}
The NLS equation \eqref{NLS} has several conserved quantities, one of them is mass (or the $L^2$ norm)
\begin{equation}\label{E:M}
M[u](t) = \int |u(x,t)|^2 \, dx  = M[u](0),
\end{equation}
which is derived by multiplying the equation \eqref{NLS} by $\bar{u}$, integrating in time and using integration by parts on the second term, then separating into real and imaginary parts, and noticing that $\partial_t \int |u(x,t)|^2 \, dx = 0$.

Another conserved quantity\footnote{The general equation \eqref{NLS} has also a conserved momentum, but we do not use it in this paper, and thus, omit it.} is the energy (or Hamiltonian)
\begin{equation}\label{E:E}
    E[u](t) = \frac{1}{2}\int|\partial_xu |^2dx -\frac{\lambda}{\alpha + 2} \int |u|^{\alpha+2}dx = E[u](0),
\end{equation}
which is derived by multiplying the equation \eqref{NLS} by $\bar{u_t}$, taking the real part, and then applying an integration by parts.

\subsection{Useful nonlinear inequalities}
This section outlines a series of inequalities that are used throughout the nonlinear estimates in the paper. In particular, we discuss a variety of {\it difference} estimates that are fundamental for the fixed-point argument, which is an essential tool in proving a contraction, implying then the existence and uniqueness of solutions in nonlinear equations. 
We start with showing how to pull $L^\infty$ terms out of $L^2$ norm, needed later in the second method for the $\mathcal X$ space estimates.

{\begin{remark}[\textbf{M.II}] For $f,g \in \mathcal X$, we have to estimate expression of type $\|f \partial_x^\beta g\|_{L^2}$ for two ranges of $\beta$ as in parts II and III in space $\mathcal X$: $1 \leq \beta \leq r$ and $r+1 \leq \beta \leq M$. Here is a general strategy for that:

$$ 
(a) \hspace{.3cm} \|f \partial^\beta_x g\|_{L^2} = \|\la x\ra^{-n} (\la x\ra^n   f) \, \la x\ra^{-n} (\la x\ra^{n} \partial^\beta_x  g) \|_{L^2}
\leq \|\la x\ra^{-2n}\|_{L^\infty} \|\la x\ra^n  f \|_{L^\infty} \| \la x\ra^{n} \partial^\beta_x  g \|_{L^2}
\leq \| f \|_{\mathcal X} \| g \|_{\mathcal X},
$$ 
for $1 \leq \beta \leq r$,
which shows that the weight term is trivially bounded: $\|\la x\ra^{-2n}\|_{L^2_x} < 1$, $n>0$, and no additional restriction on $n$ is needed. 
$$
(b) \hspace{.5cm} \|f \partial^\beta_x g\|_{L^2} = \|\la x\ra^{-n} (\la x\ra^n   f) \, (\partial^\beta_x  g) \|_{L^2}
\leq \|\la x\ra^{-n}\|_{L^\infty} \|\la x\ra^n  f \|_{L^\infty} \| \partial^\beta_x  g \|_{L^2}
\leq \| f \|_{\mathcal X} \| g \|_{\mathcal X}, \hspace{2.3cm}
$$
for $r+1 \leq \beta  \leq M$, which, similarly to the previous case, does not give any additional restrictions on $n$.\\
\end{remark}}
We now show several inequalities for the differences of nonlinearities, since those are needed in both arguments to show that the mappings considered are contractions. 

\begin{lemma}
Let $u, v \in C^\beta(\R)$ and $\alpha, \beta > 0$. Then the following estimates hold
\begin{align}\label{meanvalueineq}
| |u|^{\alpha}u - |v|^{\alpha}v | 
& \leq \tfrac{2\alpha + 3 }{2}(|u|^{\alpha} + |v|^{\alpha})|u-v|, \\
\label{deriv_difference}
\left|\partial_x \left(|u|^{\alpha}u - |v|^{\alpha }v\right)\right| 
&\leq C(|u|^\alpha + |v|^\alpha)|\partial_x(u-v)| + C (|u|^{\alpha-1}|\partial_xu| + |v|^{\alpha-1}|\partial_xv|)|u-v|.
\end{align}
\emph{(\textbf{M.II})} If, additionally, $u$ and $v$ satisfy the infimum condition \eqref{inf}, then
(non-symmetrically) we have 
{\small    
\begin{equation}\label{derivdiffinf}
|u|^\alpha\partial^\beta_x u - |v|^\alpha \partial^\beta_x v 
\leq C\Big(|u|^{\alpha}|\partial^\beta_xu-\partial^\beta_xv|+\big(\eta^{-|\alpha-1|}+ (|\langle x \rangle^n u|^{|\alpha-1|} + |\langle x \rangle^n v|^{|\alpha-1|}) \big) \langle x \rangle^{-n(\alpha-1)}|\partial^\beta_x v||u-v|\Big),
\end{equation}
}
or, written symmetrically, 
{\small
\begin{align}\label{derivdiffinf_sym}
\big| |u|^\alpha\partial^\beta_x u - |v|^\alpha \partial^\beta_x v \big|
& \leq C\big( (|u|^{\alpha}+ |v|^{\alpha}) |\partial^\beta_xu-\partial^\beta_xv|\\
&+\big(\eta^{-|\alpha-1|}+ (|\langle x \rangle^n u|^{|\alpha-1|} + |\langle x \rangle^n v|^{|\alpha-1|}) \big) \langle x \rangle^{-n(\alpha-1)}(|\partial^\beta_x u|+|\partial^\beta_x v|)|u-v|\big).\notag
\end{align}
}
Furthermore, for any positive integer $k$, we have
{\small
\begin{equation}\label{difference2k}
  \begin{aligned}
    \big||u|^{\alpha -2k} - |v|^{\alpha -2k}\big| 
    &\leq  C\left(\eta^{2+4k} \langle x \rangle^{2n(1+2k)} (|u|^{2k}+ |v|^{2k}) 
    |u|^{\alpha+1} + \eta^{1+2k}\langle x \rangle^{n(1+2k)}(|u|^\alpha + |v|^\alpha ) \right) |u-v|.
    \end{aligned}
\end{equation}
}
\end{lemma}
\begin{proof}[Proof of \eqref{meanvalueineq}]
The inequality itself is straightforward, for the constant, see \cite[Section 3, Page 14]{Vladimirov1987}. 
\end{proof}
\begin{proof}[Proof of \eqref{derivdiffinf}]
We begin by adding and subtracting the term $|u|^\alpha\partial_x^\beta v$ to the difference,
\begin{align}\label{derivdiffinf_step1}
    \Big||u|^\alpha\partial^\beta_x u - |v|^\alpha \partial^\beta_x v\Big| &= \Big||u|^\alpha\partial^\beta_x u - |u|^\alpha\partial_x^\beta v + |u|^\alpha\partial_x^\beta v - |v|^\alpha \partial^\beta_x v\Big| \\\label{derivdiffinf_step2}
    &= \Big||u|^\alpha(\partial_x^\beta u - \partial_x^\beta v) + (|u|^\alpha - |v|^\alpha)\partial_x^\beta v\Big|.
\end{align}
To estimate $|u|^\alpha - |v|^\alpha$, let $f(x) = |x|^\alpha$. Then, by the mean value inequality,
\begin{equation}\label{E:f^prime}
|f(u)-f(v)|\leq C\sup_{(u,v)}|f^\prime||u-v|.
\end{equation}
Using \eqref{Xinf}, we obtain
\begin{equation*}
    |f^\prime(u)| \leq C_\alpha|u|^{\alpha - 1} \leq \begin{cases}
        C_\alpha \eta^{1-\alpha}\langle x\rangle^{-n(\alpha-1)} \text{ if } \alpha < 1, \\
        C_\alpha|u|^{\alpha - 1} \text{ if } \alpha \geq 1.\\
    \end{cases}
\end{equation*}
Putting the above estimates together, we thus deduce
\begin{equation*}
    |f^\prime(u)| \leq C_\alpha(\langle x \rangle^{-n(\alpha-1)}\eta^{1-\alpha}+ |u|^{|\alpha - 1|} ). 
\end{equation*}
Now, the supremum over $u$ and $v$ of $f^\prime$ in \eqref{E:f^prime} can be overestimated by taking the sum of both $f^\prime(u)$ and $f^\prime(v)$, to obtain 
\begin{align*}
||u|^\alpha - |v|^\alpha| &\leq C_\alpha(\langle x \rangle^{-n(\alpha-1)}\eta^{1-\alpha}+ |u|^{|\alpha - 1|} + |v|^{|\alpha-1|} )|u-v|\\ 
&\lesssim (\eta^{1-\alpha} + (|\langle x \rangle^{n}u|^{|\alpha - 1|} + |\langle x \rangle^{n}v|^{|\alpha - 1|}))\langle x \rangle^{-n(\alpha-1)}|u-v|.
\end{align*}
Inserting the above into \eqref{derivdiffinf_step2}, yields the desired result.
\end{proof}
\begin{proof}[Proof of \eqref{deriv_difference}]
{Recalling that $|u|^\alpha = u^{\frac{\alpha}2} \bar{u}^{\frac{\alpha}2}$}, we begin by differentiating the difference of two functions and grouping corresponding terms to obtain
\begin{align*}
    \left|\partial_x \left(|u|^{\alpha}u - |v|^{\alpha }v\right)\right| &\leq \left|\frac{\alpha }{2}\left( |u|^{\alpha - 2}u^2\partial_x\Bar{u} - |v|^{\alpha - 2}v^2\partial_x\Bar{v}\right) + (1 + \frac{\alpha}{2})\left(|u|^{\alpha }\partial_x u - |v|^{\alpha }\partial_x v\right)\right|  \\
    &\leq \Big|\frac{\alpha }{2}\left( |u|^{\alpha - 2}u^2\partial_x(\Bar{u} -\Bar{v}) +(|u|^{\alpha - 2}u^2 -  |v|^{\alpha - 2}v^2)\partial_x\Bar{v}\right) \\  &\hspace{1.5cm} +  (1 + \frac{\alpha}{2})\left(|u|^{\alpha }\partial_x(u-v) + (|u|^{\alpha } - |v|^{\alpha })\partial_x v\right)\Big|.
\end{align*}
Then using a mean value theorem, as in the proof of \eqref{derivdiffinf}, we obtain
$$
\qquad \qquad  ||u|^{\alpha } -  |v|^{\alpha}| \leq C(|u|^{\alpha-1} + |v|^{\alpha-1})|u-v|, \quad \mbox{and}$$
$$\underbrace{||u|^{\alpha - 2}u^2 -  |v|^{\alpha - 2}v^2|}_{||u|^{\alpha-2}(u^2-v^2) + (|u|^{\alpha-2} - |v|^{\alpha-2})v^2|}\leq C(|u|^{\alpha-1} + |v|^{\alpha-1})|u-v|.
$$
Therefore, for $\alpha > 1$ (which is an important condition!), we have
\begin{align*}
    \left|\partial_x \left(|u|^{\alpha}u - |v|^{\alpha }v\right)\right| 
    &\leq C\Big|\frac{\alpha }{2}\left( |u|^{\alpha - 2}u^2\partial_x(\Bar{u} -\Bar{v}) +(|u|^{\alpha-1} + |v|^{\alpha-1})|u-v|\partial_x\Bar{v}\right) \\  & + \hspace{0.04cm} (1 + \frac{\alpha}{2})\left(|u|^{\alpha }\partial_x(u-v) + (|u|^{\alpha-1} + |v|^{\alpha-1})|u-v|\partial_x v\right)\Big|\\
    &\leq C(\alpha+1)|u|^\alpha|\partial_x(u-v)| + C(\alpha+1)(|u|^{\alpha-1} + |v|^{\alpha-1})|u-v||\partial_xv|.
\end{align*}
Through symmetry (exchanging $u$ and $v$),  one obtains the estimate \eqref{derivdiffinf_sym}.
\end{proof}
\begin{proof}[Proof of \eqref{difference2k}] First, notice that 
    \begin{equation*}
        \big||u|^{\alpha -2k} - |v|^{\alpha -2k} \big| = \left|\frac{|u|^{\alpha+1}}{|u|^{1+2k}} - \frac{|v|^{\alpha+1}}{|v|^{1+2k}}\right|=\left||u|^{\alpha+1}\left(\frac{1}{|u|^{1+2k}} - \frac{1}{|v|^{1+2k}}\right) + \frac{1}{|v|^{1+2k}}(|u|^{\alpha+1}-|v|^{\alpha+1})\right|.
\end{equation*}
Now, through an application of \eqref{inf} and the triangle inequality, we obtain the desired result. 
\end{proof}

The following remark and lemmas provide a framework for handling the estimates of the nonlinear term $|u|^\alpha u$ for \eqref{NLS} in Section \ref{S: Weights}. 
\begin{remark}[\textbf{M.II}]\label{R:2} 
In this remark we comment about the notation for a derivative of a product, which is used throughout the paper when dealing with nonlinear terms. First, note that $|u|^{\alpha} = (u\overline{u})^{\frac{\alpha}{2}}$. Using the Leibniz's rule, we write a derivative of $\partial_{x}^\beta(|u|^{\alpha} u)$ as
\begin{equation}{\label{NLderivative}}
 \partial_{x}^\beta(|u|^{\alpha} u) = \sum_{\substack{\gamma + \rho = \beta \\ 
 \gamma \neq 0}} C_{\gamma, \rho} \partial_{x}^{\gamma}(|u|^{\alpha})\partial_{x}^{\rho}u + C_{\gamma,\beta}|u|^{\alpha}\partial_{x}^{\beta}u.  
\end{equation}
{
\begin{example}[\textbf{M.II}]\label{Ex:1}
To provide a concrete example of \eqref{NLderivative}, take $\beta = 2$, then
    \begin{align}
        \partial_{x}^2(|u|^{\alpha} u) &= \sum_{\substack{\gamma + \rho = 2 \\ 
 \gamma \neq 0}} C_{\gamma, \rho} \partial_{x}^{\gamma}(|u|^{\alpha})\partial_{x}^{\rho}u + C_{0,2}|u|^{\alpha}\partial_{x}^{2}u \notag \\
    &= C_{2, 0} \partial_{x}^{2}(|u|^{\alpha})u + C_{1, 1} \partial_{x}(|u|^{\alpha})\partial_{x}u + C_{0,2}|u|^{\alpha}\partial_{x}^{2}u.\label{R:2_example}
    \end{align}
    Using the classic product rule, we obtain
    \begin{equation*}
        \partial_{x}(|u|^{\alpha} u) = \partial_x(|u|^\alpha) u + |u|^\alpha \partial_x u,
    \end{equation*}
    and applying it to the second derivative, we get
    \begin{align*}
        \partial_{x}^2(|u|^{\alpha} u) &= \partial_x\big(\partial_x(|u|^\alpha) u + |u|^\alpha \partial_x u\big) \\
        & = \partial^2_x(|u|^\alpha) u + 2\partial_x(|u|^\alpha) \partial_x u + |u|^\alpha \partial^2_x u,
    \end{align*}
    which matches \eqref{R:2_example} with $C_{2,0} = 1, \; C_{1,1} = 2, \; C_{0,2} = 1.$
\end{example}
}
Using similar computations (and separating into $|u|^2$ terms), a $\partial_x^\gamma$ derivative of $|u|^{\alpha}$ is written as
\begin{equation}{\label{GeneralNLderivative}}
    \partial_{x}^{\gamma}(|u|^{\alpha}) = \sum_{k=1}^{\gamma}|u|^{\alpha - 2k}\bigg(\sum_{\substack{\ell_{1}+...+\ell_{k} = \gamma \\ \ell_{j} \geq 1}} C_{\ell_{1}+...+\ell_{k}} \, \partial_{x}^{\ell_{1}}(|u|^{2})...\partial_{x}^{\ell_{k}}(|u|^{2})\bigg).  
\end{equation}
Our next step is to expand the derivative of $\partial^{\ell_j}_x(|u|^2)$ for each $j$, first noting that 
\begin{equation}\label{ProductRule}
    \partial_{x}^{\ell_{j}}(|u|^{2}) =  \sum_{m = 0}^{\ell_j}\binom{\ell_j}{m}\partial_x^{\ell_j - m}u \partial_x^{m}\bar{u} 
\end{equation}
Then for $\ell_1$ and $\ell_2$, the product may be expressed as
\begin{align*}
    \partial_{x}^{\ell_{1}}(|u|^{2})\partial_{x}^{\ell_{2}}(|u|^{2}) &=  \sum_{\substack{m_{1,1} + m_{2,1} = \ell_1\\m_{1,2} + m_{2,2} = \ell_2}}\partial^{m_{1,1}}u\partial^{m_{2,1}}\bar{u}\partial^{m_{1,2}}u\partial^{m_{2,2}}\bar{u},
\end{align*}
where $m_{1,j} + m_{2,j} = \ell_j$ corresponds to all combinations of $m_{1,j}$ and $m_{2,j}$, which add up to $\ell_j$. Putting together, we therefore write
\begin{align}\label{E:Sum_of_products}
    \partial_{x}^{\ell_{1}}(|u|^{2})...\partial_{x}^{\ell_{k}}(|u|^{2}) = \sum_{\substack{m_{1,j} + m_{2,j} = \ell_j\\ j = 1,...,k}}\prod_{j=0}^k\partial^{m_{1,j}}u\partial^{m_{2,j}}\bar{u}.
\end{align}
By substituting  \eqref{GeneralNLderivative} and \eqref{E:Sum_of_products} into \eqref{NLderivative}, then simplifying the products, we obtain
{\small
\begin{equation}{\label{genchain}}
    \begin{aligned}
        \partial_{x}^{\beta}(|u|^{\alpha}u) = C_{0,\beta}|u|^{\alpha}\partial_{x}^{\beta}u
        + \hspace{-.4cm} {\footnotesize 
        \sum_{\substack{\gamma + \rho = \beta \\ \gamma \neq 0}} \sum_{k=1}^\gamma\sum_{\substack{\ell_{1}+...+\ell_{k} = \gamma \\ \ell_{j} \geq 1}}  \sum_{\substack{m_{1,j} + m_{2,j} = \ell_j\\ j = 1,...,k}}C
        }
        |u|^{\alpha - 2k}\prod_{j=0}^k\partial_x^{m_{1,j}}u \partial_x^{m_{2,j}}\bar{u}\partial_{x}^{\rho}u.
    \end{aligned}
\end{equation}
}
\end{remark}

\begin{lemma}[\textbf{M.II.}]\label{L:Difference_derivative}
    Let $\alpha > 0$ and $u,v \in \mathcal{X}$ as defined in \eqref{Xspace}. Then for $\beta \in \mathbb N$
    \begin{equation}\label{E:nonlinear_diff_deriv}
        \left| \partial_x^\beta\big((|u|^\alpha u - |v|^\alpha v) \big)\right| \leq A + B, 
    \end{equation}
    where
    \begin{align}
        {A = } 
        &C_{0,\gamma}\left| |u|^\alpha\partial_x^\beta u - |v|^\alpha \partial_x^\beta v \right|, \label{E:A}\\
        {B =}&C_{0,\rho}\left(\sum\right)^4
        \bigg||u|^{\alpha - 2k} \partial_x^\rho u \prod_{j=1}^k\partial_x^{m_{1,j}}u\partial_x^{m_{2,j}}\bar{u}- |v|^{\alpha - 2k} \partial_x^\rho v \prod_{j=1}^k\partial_x^{m_{1,j}}v\partial_x^{m_{2,j}}\bar{v}\bigg| \label{E:B},
    \end{align}
and 
{\small
\begin{equation}\label{sigma4}
        \left(\sum\right)^4 = \sum_{\substack{\gamma+\rho = \beta \\ \gamma \geq 1}} \; \sum_{k=1}^{\gamma} \; \sum_{\substack{\ell_1 + ... + \ell_k = \gamma \\ \ell_j \geq 1}} \; \sum_{\substack{m_{1,j} + m_{2,j} = \ell_j\\ j = 1,...,k}}.
    \end{equation}
}
\end{lemma}
\begin{corollary}[\textbf{M.II.}]\label{Cor:B_splitting}
Furthermore, expanding \eqref{E:B}, we have
    \begin{equation}\label{SplitB}
        B = C_{0,\rho}\left(\sum \right)^4|{B}_1|  + C_{0,\rho}\left(\sum \right)^4|{B}_2| +  C_{0,\rho}\left(\sum \right)^4|{B}_3|,
    \end{equation}
where
{\small
    \begin{align}
        {B}_1 &= \left(|u|^{\alpha -2k} - |v|^{\alpha -2k} \right)\partial_{x}^{\rho} u \prod_{j=1}^{k}\partial_{x}^{m_{1,j}}u\partial_{x}^{m_{2,j}}\bar{u}, \label{E:B1}\\ 
        {B}_2 & = \left(\partial_x^\rho u - \partial_x^\rho v \right)|v|^{\alpha - 2k} \prod_{j=1}^k \partial_x^{m_{1,j}}u \partial_x^{m_{2,j}}\bar{u}, \label{E:B2}\\
        B_3 & = |v|^{\alpha-2k}\partial_x^\rho v\left( \prod_{j=1}^k\partial_x^{m_{1,j}}u\partial_x^{m_{2,j}}\bar{u} - \prod_{j=1}^k\partial_x^{m_{1,j}}v\partial_x^{m_{2,j}}\bar{v}\right). \label{E:B3_unsplit}
    \end{align} 
}
\end{corollary}
\begin{corollary}[\textbf{M.II.}]\label{Cor:B3_estimate}
The difference of products, as seen in \eqref{E:B3_unsplit}, can be expressed as
{\small
\begin{align}\label{complex_difference_of_products}
        & \prod_{i=1}^k \partial^{m_{1,i}}u \partial^{m_{2,i}}\bar u - \prod_{i=1}^k \partial^{m_{1,i}} v \partial^{m_{2,i}}\bar v =\notag \\ & \frac{1}{2}\sum_{i=1}^k \left( \prod_{j=1}^{i-1} \partial^{m_{1,j}}v  \partial^{m_{2,j}}\bar v \right) 
        (\partial^{m_{1,i}}u  - \partial^{m_{1,i}} v)( \partial^{m_{2,i}}\bar u +\partial^{m_{2,i}}\bar v)
        \left( \prod_{j=i+1}^{k} \partial^{m_{1,j}}u \partial^{m_{2,j}}\bar u \right) \notag\\
        + & \frac{1}{2}\sum_{i=1}^k \left( \prod_{j=1}^{i-1} \partial^{m_{1,j}}v  \partial^{m_{2,j}}\bar v \right)
        (\partial^{m_{1,i}}u  + \partial^{m_{1,i}} v)( \partial^{m_{2,i}}\bar u -\partial^{m_{2,i}}\bar v)
        \left( \prod_{j=i+1}^{k} \partial^{m_{1,j}}u \partial^{m_{2,j}}\bar u \right).
    \end{align}
}    
Thus, it suffices to estimate only the part $B_3'$ to obtain the entire estimate for $B_3$ in \eqref{E:B3_unsplit}: 
{\small
\begin{align} \label{E:B3}
    B_3' = 
    |v|^{\alpha-2k}\partial_x^\rho v \Bigg[\sum_{i=1}^k \left( \prod_{j=1}^{i-1} \partial^{m_{1,j}}v \partial^{m_{2,j}}\bar v \right) 
    (\partial^{m_{1,i}}u  - \partial^{m_{1,i}} v)( \partial^{m_{2,i}}\bar u +\partial^{m_{2,i}}\bar v)
    \left( \prod_{j=i+1}^{k} \partial^{m_{1,j}}u \partial^{m_{2,j}}\bar u \right)\Bigg].
    \end{align}
}
\end{corollary}
\begin{proof}
The proof for Lemma \ref{L:Difference_derivative} and Corollary \ref{Cor:B_splitting} are consequences of the generalized product rule and careful organization of terms. The proof for Corollary \ref{Cor:B3_estimate} follows by an application of
{\small
\begin{equation}\label{difference_of_products}
        \prod_{i=1}^n a_i - \prod_{i=1}^n b_i = \sum_{k=1}^n \left( \prod_{i=1}^{k-1} b_i \right) (a_k - b_k) \left( \prod_{i=k+1}^{n} a_i \right),
\end{equation}
}
and then setting $a_i =  \partial^{m_{1,i}}u \partial^{m_{2,i}}\bar u$, $b_i =  \partial^{m_{1,i}}v \partial^{m_{2,i}}\bar v $, and using the relation 
$$
u \bar u - v \bar v = \tfrac{1}{2}((u - v)(\bar u + \bar v) +  (u + v)(\bar u - \bar v)).
$$
\vspace{-.5cm}\end{proof}

We are now ready to embark on showing the well-posedness with either of the methods.

\section{Method I: solutions via Strichartz estimates} \label{S:Strichartz}

In this section we are going to prove Theorems \ref{T:1} and \ref{T:2} for the IVP \eqref{NLS} using the first method, aka Strichartz estimates. We begin by writing the integral formulation of \eqref{NLS}, often referred to as Duhamel's formula,  
\begin{equation}\label{ClassicIntEq}
    u(t) = e^{it\partial_{x}^{2}}u_{0} + i \int_{0}^{t} e^{i\partial_{x}^{2}(t-s)}(\lambda|u|^{\alpha}u)(s) \,ds,
\end{equation}
where $e^{it\partial_x^2}u_0$ is the solution to the linear Schr\"{o}dinger equation (as discussed in Section \ref{S: Prelim}). 
The strategy is to start with the initial datum $u_0$ and to evolve it with the integral equation \eqref{ClassicIntEq}. Therefore, we set up a map and a space to apply a fixed point theorem showing that our map is a contraction, which will give us the unique solution.

\subsection{Existence and uniqueness via Strichartz estimates.}
It should be noted that \eqref{ClassicIntEq} does not require differentiability of $u(x,t)$. This, along with the isometry of the Schr\"odinger group in $L^2_x$ \eqref{L2Isometry} and the mass conservation \eqref{E:M}
hints at the possibility of some sort of well-posedness result in $L^2_x$. A classical approach would involve a fixed point argument involving a ``contraction mapping" on an operator that is given 
by \eqref{ClassicIntEq}, namely, 
\begin{equation}\label{E:Duhamel}
 \Phi[u](t) = e^{it\partial^2_x}u_0 + i \lambda \int_0^t e^{i(t-s)\partial^2_x}(|u|^{\alpha}u)(s)ds.
\end{equation} 
Then to show that this operator defines a contraction map, we must first show that $\Phi[u]$ maps some function space $\mathcal{X}$ into itself. The main question is: {\it How do we define the space $\mathcal{X}$? }

\subsubsection{$L^2$ theory}\label{L2Strich}
We have already established the fact that the space $L^2$ is a natural choice for the spatial variable (for instance, isometry \eqref{E:isom}), therefore, we start with a simple case of the sup-norm (or equivalently, the $L^\infty$ norm) in time. This leads us to consider the space $C_tL^2_x$ (since our solution is continuous in time) with the norm $\| \|\cdot\|_{L^2_x}\|_{L^\infty_t} \equiv \|\cdot\|_{L^\infty_tL^2_x}$, in other words, taking the admissible pair $(\infty,2)$  allows us to use Strichartz estimates in Theorem \ref{T:StrichartzEstimates}. To show that $\Phi: C([0,T],L_x^2(\R)) \mapsto C([0,T],L_x^2(\R))$, it suffices to deduce that for any $f \in C([0,T],L_x^2(\R))$, we have 
$$
\|\Phi[f]\|_{L^\infty_t L_x^2} \leq C \|f\|_{L^\infty_t L_x^2},
$$
here, $L^\infty_t$ means that the sup norm is taken in time $t$ over $t \in [0,T]$.    

Thus, we consider $u_0 \in L^2(\R)$ and show that we can construct $u \in C([0,T],L_x^2(\R))$ via a Picard iteration, so that $\Phi[u]$ is bounded in $L^\infty_t L^2_x$.   
We proceed by taking the $L^\infty_tL^2_x$ norm of both parts in \eqref{E:Duhamel}, splitting the terms via the triangle inequality,  then on the first term we apply isometry \eqref{E:isom} (a special case of the linear Strichartz estimate \eqref{groupprop1}) and on the second term we use Minkowski's inequality to bring the $L^2_x$ norm inside of the integral to get

\begin{align}
    \|\Phi[u]\|_{L^\infty_tL^2_x} &\leq \|e^{it\partial^2_x}u_0\|_{L^\infty_tL^2_x} + |\lambda|\Big\|\int_0^te^{i(t-s)\Delta}\big(|u|^\alpha u \big)(s)ds \Big\|_{L^\infty_tL^2_x} \notag\\ 
    &\leq \|u_0\|_{L^2_x} + |\lambda| \sup_{t \in [0,T]} \int_0^t \| e^{i (t-s)\partial_x^2} (|u|^\alpha u)(s) \|_{L^2_x} \, ds \label{E:2}\\
&    \leq \|u_0\|_{L^2_x} + |\lambda| \int_0^T \| u(s) \|_{L^{2(\alpha+1)}_x}^{\alpha+1} \, ds, {\small ~~\mbox{by~~isometry}~\eqref{E:isom}},  \notag \\
    &\leq \|u_0\|_{L^2_x} + |\lambda| \,  T^{\frac{4-\alpha}{4}} \|u\|^{\alpha + 1}_{L^{\frac{4(\alpha + 1)}{\alpha}}_{[0,T]} L^{2(\alpha + 1)}_x}, \label{E:3}
\end{align}
where in the last step we applied the H\"older inequality in time (with $\frac{4}{\alpha}$ and $\frac4{4-\alpha}$). Note that the pair\footnote{Here, we show an approach with the specific pair $(\frac{4(\alpha+1)}{\alpha}, 2(\alpha+1))$, however, one may notice that other pairs could also work, if we bound the last term in \eqref{E:2} not by $L^1_tL^2_x$ but by $L^{q^\prime}_t L^{r^\prime}_x$ with another admissible pair $(q,r)$. 
}   
$\big(\frac{4(\alpha + 1)}{\alpha}, 2(\alpha+1)\big)$ is admissible for any $\alpha>0$, though to keep the power of $T$ positive, we need $\alpha<4$. 
From \eqref{E:3} it follows that we now need to estimate $\Phi[u]$ in {\small $L^{\frac{4(\alpha + 1)}{\alpha}}_{t} L^{2(\alpha + 1)}_x$}, which we do in a similar manner
\begin{align}
\|\Phi[u]\|_{L^{\frac{4(\alpha + 1)}{\alpha}}_{t} L^{2(\alpha + 1)}_x} &\leq \|e^{it\partial^2_x}u_0\|_{L^{\frac{4(\alpha + 1)}{\alpha}}_{t} L^{2(\alpha + 1)}_x} + |\lambda|\Big\|\int_0^te^{i(t-s)\Delta}\big(|u|^\alpha u \big)(s)ds \Big\|_{L^{\frac{4(\alpha + 1)}{\alpha}}_{t} L^{2(\alpha + 1)}_x} \notag\\ 
    &\leq C \, \|u_0\|_{L^2_x} + C \, |\lambda| \| |u|^\alpha u\|_{
    L_t^{q_0'}L^{r_0'}_x}  ~~ {\small \mbox{by} ~~ \eqref{E:Str}~~ \mbox{for~~ any ~~admissible} ~~ (q_0,r_0)}, \notag \\
    & \leq C \, \|u_0\|_{L^2_x} + C \, |\lambda| \| u \|_{L^{q_0'(\alpha+1)}_t L^{r_0'(\alpha+1)}_x}^{\alpha+1}. \label{E:4}
\end{align}
Then taking $r_0' = 2$ (to match the $L^{2(\alpha+1)}_x$ space), and hence, $q_0' = 1$ (thus, $r_0 = 2$ and $q_0 = \infty$), we use the same H\"older inequality in time to bound the last term in \eqref{E:4}  by $C \,|\lambda| \,  T^{\frac{4-\alpha}{4}} \|u\|^{\alpha + 1}_{L^{\frac{4(\alpha + 1)}{\alpha}}_{t} L^{2(\alpha + 1)}_x}$, which is the same as in \eqref{E:3}. 

Summarizing \eqref{E:3} and \eqref{E:4}, for $\theta = \frac{4 - \alpha}{4}$ (and $C$ a fixed constant from Strichartz estimates), we have 
\begin{equation}\label{E:L2_LqLrest1}
\|\Phi[u]\|_{L^\infty_tL^2_x} \leq \|u_0\|_{L^2_x} + |\lambda| \,T^\theta\|u\|^{\alpha + 1}_{L^{\frac{4(\alpha+1)}{\alpha}}_t L^{2(\alpha+1)}_x} \end{equation}
and
\begin{equation}\label{E:L2_LqLrest2}
\|\Phi[u]\|_{L^{\frac{4(\alpha+1)}{\alpha}}_t L^{2(\alpha+1)}_x} 
    \leq C \, \|u_0\|_{L^2_x} + C \, |\lambda| \, T^\theta\|u\|^{\alpha+1}_{L^{\frac{4(\alpha+1)}{\alpha}}_t L^{2(\alpha+1)}_x}.
\end{equation}

Considering both \eqref{E:L2_LqLrest1} and \eqref{E:L2_LqLrest2}, one can notice that if we take $a>0$ such that $2 C \|u_0\|_{L^2} = {a}$  and $\|u\|_{L^{\frac{4(\alpha+1)}{\alpha}}_t L^{2(\alpha+1)}_x}$ would stay bounded by $a$,
then the right-hand side of both \eqref{E:L2_LqLrest1} and \eqref{E:L2_LqLrest2} will be bounded by $\frac{a}2 + C\, |\lambda| \, T^\theta a^{\alpha+1} = a( \frac12 + C\, |\lambda| \, T^\theta a^\alpha)$. Then we could choose $T>0$ small enough so that $C\, |\lambda| \, T^\theta a^{\alpha} < \frac12$, and thus, the last expression would be bounded by $a(\frac12 + \frac12) = a$. This would show that $\Phi$ maps the space, where both $L^{\frac{4(\alpha+1)}{\alpha}}_{[0,T]}L^{2(\alpha+1)}_x$ and $L^\infty_{[0,T]} L^2_x$ norms are bounded by $a$, into itself.
Thus, at this point we have ``closed the loop" for our estimates and have established a suitable function space for $u(x,t)$ to ``live" in, namely,  
\begin{equation}\label{E:space} 
C([0,T],L^2(\mathbb{R})) \cap L^q([0,T],L^r(\R))
\end{equation}
with $(q,r)$ being an admissible pair $( \frac{4(\alpha+1)}{\alpha}, 2(\alpha+1) )$. 

Now, for any $a>0$ and $T>0$, set $q=\frac{4(\alpha+1)}{\alpha}$ and $r=2(\alpha+1)$, and define
\begin{equation}\label{L2StrichSpace}
    {X}_{a,T} = \big\{ u \in C([0,T],L^2(\mathbb{R})) \cap L^{q}([0,T],L^{r}(\R)) \; : \;\vertiii{ u}_{\mathcal{X}_{a,T}} \leq a \big\},
\end{equation}
where
\begin{equation}\label{L2StrichSpaceNorm}
\vertiii{u}_{\mathcal{X}_{a,T}} \stackrel{def}= \sup_{t \in [0,T]}\| u(t) \|_{L^2_x(\mathbb{R})} + \left( \int_{0}^T \|u(t)\|_{L^{r}}^{q} dt \right)^{\frac{1}{q}}.
\end{equation}
Given $u_0 \in L^2(\mathbb R)$ and $\alpha<4$, set $a = 2C \|u_0\|_{L^2}$, where $C$ is the maximum of Strichartz constants from \eqref{groupprop1} and \eqref{E:Str} for pairs $(\infty,2)$ and $\big( \frac{4(\alpha+1)}{\alpha}, 2(\alpha+1)\big)$.  Then take $T>0$ so that
\begin{equation}\label{E:T}
\qquad \qquad T  < \frac{1}{ \big( 2C \, |\lambda| \, a^\alpha \big)^{\frac4{4-\alpha}} } \equiv \frac{1}{ \big( (2C)^{4+\alpha} |\lambda|\big)^{\frac{4}{4-\alpha}} \|u_0\|_{L^2}^{\frac{4\alpha}{4-\alpha}} }.
\end{equation}
Thus, if $u_0 \in \mathcal X_{a,T}$, then for $0< t \leq T$, $\Phi [u](t) \in \mathcal X_{a,T}$. Thus, $\Phi$ maps the space $\mathcal X_{a,T}$ into itself. Note that the time of existence $T$ is inversely proportional to the mass of the solution. (We note that the same argument will work for negative times, so one can take the interval $[-T,T]$ in \eqref{E:space} or \eqref{L2StrichSpace}.)
\smallskip

The next step is to show that $\Phi$ defines a {\it contraction} map on $\mathcal{X}_{a,T}$. For that we need to study the difference $\|\Phi[u] - \Phi[v]\|_{\mathcal{X}_{a,T}}$ and show that it is smaller than $ \kappa \|u-v\|_{\mathcal{X}_{a,T}}$ for a constant $0 < \kappa < 1$. Similar to the \eqref{E:2} derivation, applying the difference \eqref{meanvalueineq} and an application of the H\"older inequality in space,
we obtain 
\begin{align}
    \|\Phi[u] - \Phi[v]\|_{L^\infty_tL^2_x} &\leq |\lambda| \, \Big\|\int_0^te^{i(t-s)\Delta}\big(|u|^\alpha u - |v|^\alpha v\big)(s) ds \Big\|_{L^\infty_t L^2_x} \notag 
    \leq |\lambda| \, \||u|^\alpha u - |v|^\alpha v\|_{L^{1}_tL^{2}_x} \notag \\
    &\leq  |\lambda| \int_0^T\left|\||u(s)|^\alpha + |v(s)|^\alpha\|_{L^{\frac{2(\alpha+1)}{\alpha}}_x}\|u(s)-v(s) \|_{L^{2(\alpha+1)}_x} \right| ds \notag \\
    &\leq  |\lambda| \int_0^T\left(\|u(s)\|_{L^{2(\alpha+1)}_x}^\alpha + \|v(s)\|^\alpha_{L^{2(\alpha+1)}_x}\right)\|u(s)-v(s)\|_{L^{2(\alpha+1)}_x}ds.
    \label{E:L2_difference}
\end{align}
From here we have to be careful with our application of the H\"older inequality, specifically,
\begin{equation}\label{3holderineq}
    \int_0^T \left(f^\alpha \cdot g \cdot 1\right) dt\leq \left(\int_0^T f^{\alpha p}dt\right)^{1/p} \left(\int_0^T g^{q}dt\right)^{1/q} \left(\int_0^T dt\right)^{1/h},
\end{equation}
where $1 = \frac{1}{p} + \frac{1}{q} + \frac{1}{h}$. Choosing $p = \frac{4(\alpha+1)}{\alpha^2}$ and $q = \frac{4(\alpha+1)}{\alpha}$, we obtain $\frac{1}{h} = \frac{4-\alpha}{\alpha}$, which matches our definition of $\theta$ in \eqref{E:L2_LqLrest1} and \eqref{E:L2_LqLrest2}. Applying \eqref{3holderineq} to \eqref{E:L2_difference}, we deduce by definition of the space \eqref{L2StrichSpace},
\begin{align}
    \|\Phi[u] - \Phi[v]\|_{L^\infty_tL^2_x} 
    &\leq  |\lambda|\, T^\theta \bigg( \|u\|_{L^{\frac{4(\alpha+1)}{\alpha}}_t L^{2(\alpha+1)}_x}^\alpha + \|v\|^\alpha_{L^{\frac{4(\alpha+1)}{\alpha}}_tL^{2(\alpha+1)}_x} \bigg) \|u-v\|_{L^{\frac{4(\alpha+1)}{\alpha}}_t L^{2(\alpha+1)}_x} \notag \\
    &\leq 2 |\lambda| \, T^\theta a^{\alpha}\|u-v\|_{L^{\frac{4(\alpha+1)}{\alpha}}_tL^{2(\alpha+1)}_x},\label{E:L2_difference_cont}
\end{align}
and similarly, following the derivation in \eqref{E:4}, 
\begin{align}
    \|\Phi[u] - \Phi[v]\|_{L^{\frac{4(\alpha+1)}{\alpha}}_tL^{2(\alpha+1)}_x} 
    &\leq C |\lambda| \, T^\theta\left(\|u\|_{L^{\frac{4(\alpha+1)}{\alpha}}_t L^{2(\alpha+1)}_x}^\alpha + \|v\|^\alpha_{L^{\frac{4(\alpha+1)}{\alpha}}_tL^{2(\alpha+1)}_x}\right) \|u-v\|_{L^{\frac{4(\alpha+1)}{\alpha}}_tL^{2(\alpha+1)}_x} \notag \\
    &\leq 2C |\lambda| \, T^\theta a^{\alpha}\|u-v\|_{L^{\frac{4(\alpha+1)}{\alpha}}_tL^{2(\alpha+1)}_x}.
    \label{E:L2_difference_LqLr}
\end{align}

Then by choosing $T>0$ (possibly even smaller than in \eqref{E:T}), and such that  $\kappa = 2CT^\theta a^{\alpha} < 1$, we conclude that $\Phi: \mathcal{X}_{a,T} \righttoleftarrow$ is a contraction, {to be precise $\|\Phi[u] - \Phi[v]\|_{\mathcal{X}_{a,T}} \leq \kappa\|u-v\|_{\mathcal{X}_{a,T}}$ with $\kappa <1$. 
Hence, by the Banach fixed-point theorem 
there exists a unique $u^*\in \mathcal{X}_{a,T}$ such that $\Phi[u^*] = u^*$. Recalling that $\mathcal X_{a,T} \equiv C([0,T],L^2(\mathbb{R})) \cap L^{q}([0,T],L^{r}(\R))$, we obtain that the solution $u(t)$ exists (continuously up to time $T$) in $L^2(\mathbb R)$ (with an additional property that it also belongs to $L^q_tL^p_x$) and is unique. This proves the first part of Theorem \ref{T:1}. }

\subsubsection{$H^1$ Theory}
In the previous section, we developed the $L^2$ theory for the NLS equation \eqref{NLS}-\eqref{E:one}. Motivated by mass conservation, we found that we did not need 
differentiability of the initial data $u_0$.  
However, if we would like our solution to have finite energy, and thus, exhibit energy conservation  
\eqref{E:E}, the solution must be in $H^1$. This is one of the motivations to establish a well-posedness theory in the Sobolev space $H^1(\R)$, for which we can use our previous computations in $L^2(\R)$ as a foundation. Recall that if $f \in H^1$, then $f,\partial_x f \in L^2$. To prepare for the estimates to come, we introduce an estimate for the derivative of the nonlinear term $|u|^\alpha u$ that will be useful later
\begin{equation*}
    \partial_x(|u|^\alpha u ) = \alpha |u|^{\alpha-2}(\partial_x u \bar u + u \partial_x \bar u)u = \alpha |u|^{\alpha}\partial_xu + \alpha|u|^{\alpha-2}u^2 \partial_x \bar u .
\end{equation*}
Thus, for $\alpha \geq 0$,
\begin{equation*}
    |\partial_x(|u|^\alpha u )| \leq C|u|^{\alpha} |\partial_xu| 
\end{equation*}
for a constant $C>0$ depending on $\alpha$. 

In the previous subsection, we established a variety of estimates for $\Phi$ in $L^\infty_t L^2_x$. To highlight some differences, we begin by estimating $\partial_x \Phi$ in $L^\infty_t L^2_x$ as we want to work at the $H^1$ regularity, obtaining
\begin{align}    
\|\partial_x\Phi[u]\|_{L^\infty_tL^2_x}
& \leq \|\partial_xu_0\|_{L^2_x} + |\lambda|\Big\|\int_0^te^{i(t-s) \partial_x^2} \partial_x\big(|u|^\alpha u \big)(s) ds \Big\|_{L^\infty_t L^2_x}
      \leq \|\partial_x u_0\|_{L^2_x} + |\lambda| \||u|^\alpha \partial_xu\|_{L^{1}_t L^{2}_x} \label{E:1}\\
    &\leq \|\partial_x u_0\|_{L^2_x} +  |\lambda| \int_0^T \|u(s)\|^\alpha_{L^{\ell\alpha}_x} \|\partial_x u(s)\|_{L^{r}_x} ds, \label{E:H1_LqW1r_2B}
\end{align}
with\footnote{We note that instead of $L^1_t L^2_x$ \eqref{E:1}, we could use a different admissible pair and proceed in a similar way; for clarity and conciseness we decided to fix this pair and explain how to run the argument.}$^{,}$\footnote{In \eqref{E:H1_LqW1r_2B} we could put the derivative term into $L^2$ and the rest into $L^\infty$ as $\|u\|^\alpha_{L^{\infty}_x} \|\partial_x u\|_{L^{2}_x}$, then use the 1D Sobolev embedding of $H^{\frac12 +} \hookrightarrow L^\infty$ and bound the first term in $H^1$; however, this would only work in 1D, so we decided to show a more general splitting. 
} 
$\frac{1}{2} = \frac{1}{\ell} + \frac{1}{r}$, or $l=\frac{2r}{r-2}$.  Using the embedding   
$W^{s,r} \hookrightarrow L^{\ell \alpha}$ in $\R$ with
$s \in (\frac1{r} - \frac1{\ell \alpha}, 1)$
(noting that $\frac1{r} - \frac1{\ell \alpha} = \frac1{r}(1-\frac{r-2}{2 \alpha})  \leq \frac12$, 
since $r \geq 2$ from the admissible pair condition, and also choosing $r<2(\alpha+1)$ to keep the expression in parentheses positive), we have $\|u\|_{L^{\ell\alpha}_x} \leq \|u\|_{W^{s,r}_x}$. Thus, we continue with bounding the last term in \eqref{E:H1_LqW1r_2B} by
\begin{equation}\label{E:6}
|\lambda| \int_0^T \|u(s)\|^\alpha_{W^{s,r}_x} \| u(s)\|_{W^{1,r}_x} ds  
\leq  C |\lambda| T^\theta \|u\|^{\alpha+1}_{L^{q}_t W^{1,r}_x}, 
\end{equation}
where we used the fact that $\|\partial_x f\|_{L^r_x}\leq C\|f\|_{W^{1,r}}$ for $f \in W^{1,r}(\R)$ and the trivial embedding $W^{1,r} \hookrightarrow 
W^{s,r}$ on $\mathbb R$, since $s<1$, and then applied the H\"older inequality in time, yielding $\theta = 1-\frac{\alpha+1}{q}$, for which we require $q >\alpha+1$. This is ensured, since $(q,r)$ is an admissible pair, $q = \frac{4r}{r-2} > \alpha+1$ holds automatically if $\alpha \leq 3$ and provided that we choose $r <\frac{2(\alpha+1)}{\alpha-3}$ if $\alpha >3$. At this point note that for $\alpha \leq 3$ we require $r<2(\alpha+1)$ and for $\alpha >3$ we need $r< \min\{2(\alpha+1), \frac{2(\alpha+1)}{\alpha-3} \}$, which amounts to $r<2(\alpha+1)$ for $\alpha \leq 4$ and $r <\frac{2(\alpha+1)}{\alpha-3}$ for $\alpha >4$.
Thus, we have
\begin{equation}\label{E:7}
 \|\partial_x\Phi[u]\|_{L^{\infty}_t L^2_x} 
 \leq C \|\partial_x u_0\|_{L^2_x} + C |\lambda| T^\theta \|u\|^{\alpha+1}_{L^{q}_t W^{1,r}_x}.
\end{equation}

Since we ended up with the space $L^{q}_t W^{1,r}_x$, we need to make sure we also estimate $\partial_x \Phi[u]$ in $L^{q}_t L^{r}_x$ to ``close the loop", which we do, applying linear Strichartz inequality \eqref{groupprop1} on the first term and nonlinear Strichartz \eqref{E:Str} with pairs $(q,r)$ and $(1,2)$ on the second, obtaining 
\begin{align}    
\|\partial_x\Phi[u]\|_{L^q_tL^r_x} 
    & \leq \|e^{it\partial_x^2} \partial_x u_0\|_{L^q_t L^r_x} +  |\lambda| \Big\|\int_0^t e^{i(t-s) \partial_x^2} \partial_x\big(|u|^\alpha u \big)(s) ds \Big\|_{L^q_t L^r_x} \notag  \leq C \|\partial_x u_0\|_{L^2_x} + C|\lambda| \||u|^\alpha \partial_xu\|_{L^{1}_t L^{2}_x}, \notag
\end{align} 
and then we proceed nearly identically as before, yielding (with the same $\theta$)
\begin{align}\label{E:8}
 \|\Phi[u]\|_{L^q_t W^{1,r}_x} 
 \leq C \|\partial_x u_0\|_{L^2_x} + C |\lambda| T^\theta \|u\|^{\alpha+1}_{L^{q}_t W^{1,r}_x}.
\end{align}
Notice that just like in the $L^2(\R)$ theory, we have introduced a new space (in this case $L^q_tW^{1,r}(\R)$ to fill in the gaps and close the loop). In the same way that $L^\infty_tL^2_x$ was not sufficient to deal with all nonlinear powers and $L^q_tL^r_x$ was introduced, $L^\infty_tH^1_x$ (equivalent to $L^\infty_tW^{1,2}_x$) is not sufficient. New tools available to us, such as the Sobolev embedding used going from \eqref{E:H1_LqW1r_2B} to \eqref{E:6}, helped us reduce some of the limitations we inherited from the $L^2$ theory. 
Now we proceed with defining a suitable $H^1$-type space and showing that $\Phi$ is a contraction on that space. 

For any $a>0$ and $T>0$, choose $r$ and $q$ such that  
\begin{equation}\label{E:rq}
r \in \Bigg\{ \begin{array}{ll} 
(2, 2(\alpha+1) ) &\quad \mbox{if} ~~ 0 < \alpha \leq 4 \\
\big(2, \frac{2(\alpha+1)}{\alpha-3} \big) & \quad \mbox{if} ~~~~~ \alpha >4
\end{array} \Bigg\} \qquad
\mbox{and} \qquad  q = \frac{4r}{r-2}, 
\end{equation}
so that $(q,r)$ is an admissible pair and define
\begin{equation}\label{H1StrichSpace}
    \mathcal{Y}_{a,T} = \left\{ u \in C([0,T],H^1(\mathbb{R})) \cap L^q([0,T],W^{1,r}(\R)) \; : \;\vertiii{ u}_{\mathcal{Y}_{a,T}} \leq a  \right\},
\end{equation}
where
\begin{align}\label{H1StrichSpaceNorm}
   \vertiii{u }_{\mathcal{Y}_{a,T}} &= \sup_{[0,T]}\| u \|_{H^1(\mathbb{R})} + \left( \int_{0}^T \|u\|_{W^{1,r}}^qdt\right)^{\frac{1}{q}}.
\end{align}
Given $u_0 \in H^1(\mathbb R)$ and $\alpha >0$, set $a = 2C \|u_0\|_{H^1}$ and take $T>0$ such that for $\theta = 1-\frac{\alpha+1}{q}$ 
\begin{equation}\label{E:T-H1}
\qquad \qquad T  < \frac{1}{ \big( 2C \, |\lambda| \, a^\alpha \big)^{\theta } } \equiv \frac{1}{(2C)^{\theta (1+\alpha)} |\lambda|^{\theta} \|u_0\|_{H^1}^{\theta \alpha} }.
\end{equation}
If $u_0 \in H^1(\mathbb R)$, then for any $t \in [0,T]$, a similar argument as we used in the $L^2$ theory shows that $\Phi[u](t) \in \mathcal Y_{a,T}$, or $\Phi$ maps $\mathcal Y_{a,T}$ into itself, and all is left to show is that it is a contraction on this space. (Similar to the remark in the $L^2$ theory, the same argument works for negative times, so we develop the $H^1$ theory on $[-T,T]$.)
\smallskip

The difference estimates have some resemblance to the $L^2$ theory, however, there is a new difficulty to handle small powers $\alpha$. 
Applying \eqref{deriv_difference} to the difference, we get
\begin{align*}
    \|\partial_x\Phi[u] - \partial_x\Phi[v]\|_{L^\infty_tL^2_x} &\leq |\lambda| \Big\|\int_0^te^{i(t-s) \Delta}\partial_x\big(|u|^\alpha u - |v|^\alpha v\big)(s)ds \Big\|_{L^\infty_tL^2_x} 
    \leq |\lambda|\|\partial_x(|u|^\alpha u - |v|^\alpha v)\|_{L^1_t L^2_x} \\
    &\hspace{-1.5 in}\leq |\lambda| \big( \underbrace{\|\left(|u|^{\alpha } + |v|^{\alpha }\right)\partial_x|u-v|\|_{L^1_t L^2_x}}_{A} + \underbrace{\| |u-v|(|u|^{\alpha-1} |\partial_x u | + |v|^{\alpha-1} |\partial_x v |)\|_{L^1_t L^2_x}}_B \big).
\end{align*}
Then by H\"older's inequality ($l=\frac{2r}{r-2}$) and Sobolev embedding as in \eqref{E:H1_LqW1r_2B} and \eqref{E:6}, we bound $A$ as
\begin{align}
    A
    &\leq \int_0^T \big(\|u(s)\|_{L^{\ell\alpha}_x}^\alpha + \|v(s)\|^\alpha_{L^{\ell\alpha}_x}\big) \|\partial_x(u-v)(s)\|_{L^{r}_x}ds \notag 
    \leq CT^\theta\big(\|u\|_{L^q_t W^{1,r}_x}^\alpha + \|v\|^\alpha_{L^q_tW^{1,r}_x}\big) \|u-v\|_{L^q_tW^{1,r}_x}.
    \label{E:H1_difference_A}
\end{align}
To bound $B$, note that in this approach we encounter a problem for $\alpha<1$ in the term $|u|^{\alpha-1}$ (or $|v|^{\alpha-1}$). 
\begin{remark}\label{R:3} 
This was one of the main reasons to develop the second method to deal with small powers $\alpha<1$, which we discuss in the next section. We mention that in this case one could try to apply an approach by Kato \cite{Kato1987}, which instead of estimating the derivative $\partial_x \Phi$ directly, deals with a difference quotient, then taking a limit to get an estimate for the derivative, for interested reader we refer to \cite[Theorem I \& Lemma 2.2]{Kato1987}. In general, that method also requires enough smoothness of function $x \mapsto |x|^\alpha x$ and  would not be applicable, especially, in higher dimensions.
\end{remark}
Therefore, from now on we assume that $\alpha \geq 1$. Then we proceed estimating $B$ using H\"older's inequality with $l=\frac{2r}{r-2}$ and the 1D embedding $W^{s,r} \hookrightarrow L^{\ell \alpha}$ as before
\begin{align}
    B
    &\leq \int_0^T \Big(
    \big\| |u-v|\,|u|^{\alpha-1} \big\|_{L^l_x} 
    \| \partial_x u \|_{L^r_x} 
    +\big\| |u-v|\,|v|^{\alpha-1} \big\|_{L^l_x} 
    \| \partial_x v \|_{L^r_x}  \Big)(s) ds
    \notag \\
    &
    \leq \int_0^T  \Big(
    \|u-v\|_{L^{\ell \alpha}_x} \big(\|u\|^{\alpha-1}_{L^{\ell \alpha}_x}\|\partial_xu\|_{L^r_x} + \|v\|^{\alpha-1}_{L_x^{\ell \alpha}} \|\partial_xu\|_{L^r_x} 
    \big) \Big) (s) ds \notag \\
    &\leq C \int_0^T  \Big(
    \|u-v\|_{W^{s,r}_x} \big(\|u\|^{\alpha-1}_{W^{s,r}_x} | u \|_{W^{1,r}_x} + \|v\|^{\alpha-1}_{W^{s,r}_x} \| u\|_{W^{1,r}_x} \big) \Big) (s) ds \notag \\
    &\leq C T^{\theta} \big(\|u\|_{L^q_t W^{1,r}_x}^\alpha + \|v\|^\alpha_{L^q_t W^{1,r}_x} \big) 
    \|u-v\|_{L^q_t W^{1,r}_x}, \notag
\end{align}
where $\theta = 1 - \frac{\alpha + 1}{q}$ as it was previously. Putting together the estimates for $A$ and $B$, we have
\begin{equation}\label{E:diffH1}
  \|\partial_x\Phi[u] - \partial_x\Phi[v]\|_{L^\infty_tL^2_x} 
  \leq C |\lambda| T^\theta \big( \|u\|_{L^q_t W^{1,r}_x}^\alpha + \|v\|^\alpha_{L^q_tW^{1,r}_x}\big) \|u-v\|_{L^q_t W^{1,r}_x}.
\end{equation}
Similarly, the same bound holds for $\|\partial_x\Phi[u] - \partial_x\Phi[v]\|_{L^q_tL^r_x}$. 

Choosing $T>0$ same way as in \eqref{E:T-H1} (with possibly a different constant $C$), 
we obtain
\begin{align*}
    \|\partial_x\Phi[u] - \partial_x\Phi[v]\|_{\mathcal Y_{a,T}} 
    \leq \kappa \, \|u-v\|_{\mathcal Y_{a, T}} 
\end{align*}
with $0<\kappa < 1$, concluding that $\Phi: \mathcal{Y}_{a,T} \righttoleftarrow$ is a contraction, the solution to the NLS equation \eqref{NLS}-\eqref{E:one} in the space $\mathcal Y_{a,T}$ exists and is unique. {A similar argument as at the end of subsection \ref{L2Strich} shows that we obtain a continuous up to some small time $T>0$ an $H^1$ solution $u(t)$, thus, proving the first part of Theorem \ref{T:2}.} 

\subsection{Continuous Dependence}\label{S:contdep}

To complete the local well-posedness argument (either for Theorem \ref{T:1} or \ref{T:2}), it is left to show that the solution to \eqref{NLS} is continuously dependent on the initial data, which is typically shown by a Lipschitz condition $\|\Phi(u)-\Phi(v)\|_X \leq K  \|u-v\|_X$ for some positive constant $K$ in a corresponding space $X$. For that, take two solutions $u, v$ of \eqref{NLS}-\eqref{E:one} with corresponding initial data $u_0, v_0$,  then 
\begin{equation}\label{cont_difference}
    u - v = e^{it\partial_x^2}(u_0 - v_0) + i\lambda \int_{0}^t e^{i(t-\tau)\partial_x^2}(|u|^\alpha u - |v|^\alpha v ) d\tau.
\end{equation}
In the case of the $L^2(\R)$, we utilize the difference estimates \eqref{E:L2_difference_cont} and \eqref{E:L2_difference_LqLr} to obtain
\begin{align*}
    \|u - v\|_{\mathcal{X}_{a,T}} 
    &\leq \|u_0 - v_0\|_{L^2} + 2C |\lambda|T^\theta a^{\alpha}\|u-v\|_{\mathcal{X}_{a,T}}, 
\end{align*}
and since by definition of $T$, $2CT^\theta a^{\alpha} < 1$, we deduce 
\begin{align}\label{E:CD}
    \|u - v\|_{\mathcal{X}_{a,T}} \leq K\|u_0-v_0\|_{L^2},
\end{align}
completing continuous dependence for the initial data in $L^2$. The argument for $H^1$ continuous dependence follows identically, using \eqref{E:diffH1}. 

We are now finished with the first method of obtaining the well-posedness of \eqref{NLS}-\eqref{E:one} in $L^2$ or $H^1$. The main restriction of the method is that there has to be appropriate Strichartz estimates available to bound the linear and nonlinear parts in the Duhamel's formulation. In 1D this, as we have seen is typically possible (though already problematic for $\alpha<1$), however, when one tries to do the same argument in higher dimensions, then various restrictions come up, in particular,  the nonlinearities $\alpha<1$ become a problem (also, $\alpha > \frac{4}{N-2}$ in dimensions $N\geq 3$), or Sobolev embeddings do not hold for the desired values. Thus, we show another method to obtain the well-posedness in the subset of $H^1$, which can be applied to a large variety of nonlinearities, and even to different types of nonlinearities (not necessarily power nonlinearities, but can be an analytic function, etc), and we demonstrate it on an example of combined nonlinearities. It also allows one to study slow decay solutions, as we have shown an example of polynomially decaying functions in Remark \ref{R:1}. 

\section{Method II: solutions via weighted estimates} \label{S: Weights}
\subsection{Weighted linear estimates} 
In this method we also need to start with obtaining estimates for the linear  Schr\"odinger flow, given by the equation \eqref{LinearSE}. Recalling the fundamental theorem of calculus $F(b) = F(a) +\int_a^b F'(s)ds$ and applying it to $F(t) = e^{it \partial_x^2}u$ on the time interval $[0,t]$ (and using the fact that $e^{it \partial_x^2}$ commutes with $\partial^2_x$), we have
\begin{equation}\label{E:linFTC} 
    e^{it\partial_x^2}u = u + i\int_{0}^te^{is\partial_x^2}\partial^2_x u(s) ds.
\end{equation}
The linear estimate, which is weighted in this case, is the following lemma, a version of which was originally proved  in \cite{CN2017}; here, we prove it to show how weights appear and are handled, including the induction on the power weight $m$ for this proof.  
\begin{lemma}\label{Int_weight}
Given $\psi \in \mathcal{S}(\mathbb{R})$, we have $e^{it\partial_{x}^{2}} \psi \in C([0, \infty), \mathcal{S}(\mathbb{R})) \subset C([0, \infty), L^2(\mathbb R))$. Let $u(t) = e^{it\partial_{x}^{2}} \psi$, then for any $m \in \mathbb{N} \cup \{0\}$, we have 
    \begin{equation}{\label{Rmk1}}
        \|\langle x \rangle^m e^{it\partial_x^2} \psi \|_{L^2} \leq C(1+t)^m \sum_{j=0}^m\|\langle x \rangle^{m-j}\partial_x^j \psi\|_{L^2}.
    \end{equation}
\end{lemma}
\begin{proof}
Let $u(x,t) = e^{it\partial_x^2}\psi(x)$. For $m=0$, we directly utilize the isometry \eqref{L2Isometry}. For $m=1$, we multiply the linear equation $i u_t = -\partial_x^2 u$ by $\langle x \rangle^2 \Bar{u}$ and integrate, obtaining \begin{equation}\label{E:4.1a}
\int_{-\infty}^{\infty} i\langle x \rangle^2 \Bar{u}u_t = -\int_{-\infty}^{\infty} \langle x \rangle^2 \Bar{u}\partial_x^2 u.
\end{equation}
Taking the imaginary part of the left side, we obtain  
\begin{align}
\Im\left(i\int_{-\infty}^{\infty}\langle x \rangle^2 \Bar{u}u_tdx\right) 
    &= \int_{-\infty}^{\infty}\langle x \rangle^2 \Im(i\Bar{u}u_t)dx = \int_{-\infty}^{\infty}\langle x \rangle^2 \frac{i(u_t\Bar{u} + \Bar{u_t}u)}{2i}dx = \frac{1}{2}\frac{d}{dt}\int_{-\infty}^{\infty}\langle x \rangle^2|u|^2dx \notag\\
& = \frac{1}{2}\frac{d}{dt} \|\langle x \rangle u\|_{L^2}^2 =\|\langle x \rangle u\|_{L^2} \frac{d}{dt}\|\langle x \rangle u\|_{L^2}. \label{E:LHS2}
\end{align}
Using integration by parts, also taking the imaginary part, and then applying the Cauchy-Schwartz inequality, we get 
\begin{align}
-\Im\left(\int_{-\infty}^{\infty}\langle x \rangle^2 \Bar{u}\partial_x^2 udx\right) &= \underbrace{\Im\int_{-\infty}^{\infty}\langle x \rangle^2 |\partial_x u|^2 dx}_{0} + \Im\int_{-\infty}^{\infty}\partial_x(\langle x \rangle^2) \Bar{u}\partial_x u dx = \Im\int_{-\infty}^{\infty}\partial_x(\langle x \rangle^2) \Bar{u}\partial_x u dx \notag\\
 &\leq 2\int_{-\infty}^{\infty} \langle x \rangle \partial_xu \Bar{u} \, dx
        \leq 2 \|\langle x \rangle u\|_{L^2} \underbrace{\|\partial_x u \|_{L^2}}_{\|\partial_x \psi \|_{L^2} \text{ by } \eqref{L2Isometry}}. \label{E:RHS2}
\end{align}
Since $\|\langle x \rangle u\|_{L^2} \neq 0$ (otherwise, nothing to prove), recalling the equation \eqref{E:4.1a} and thus, putting together \eqref{E:LHS2} with \eqref{E:RHS2}, and integrating both sides in time, we deduce
    \begin{align*}
        \|\langle x \rangle u\|_{L^2} \leq \|\langle x \rangle \psi\|_{L^2} + 2\int_0^t\|\partial_x \psi \|_{L^2}ds 
        = \|\langle x \rangle \psi\|_{L^2} + 2t\|\partial_x \psi \|_{L^2} 
        \leq C(1+t)\sum_{j=0}^1\|\langle x \rangle^{1-j} \partial_x^j\psi\|_{L^2}.
    \end{align*}
Continuing with an induction on $m$, we assume that 
\begin{equation}\label{E:m}
\|\langle x \rangle^m u \|_{L^2} \leq C(1+t)^m \sum_{j=0}^m\|\langle x \rangle^{m-j}\partial_x^j \psi\|_{L^2}
\end{equation}
holds true and we show that a corresponding expression holds for $m+1$. Similarly to \eqref{E:4.1a}, we multiply the linear equation by $\langle x \rangle^{2(m+1)} \Bar{u}$ then, integrating by parts, applying Cauchy-Schwartz, and estimating the derivative of the weight as in \eqref{E:LHS2} and \eqref{E:RHS2}, we obtain 
\begin{align*}
        \Im\Big(i\int_{-\infty}^{\infty} u_t\Bar{u}\langle x \rangle^{2(m+1)}dx \Big) 
        &= -\Im\Big(\int_{-\infty}^{\infty} \partial_x^2 u \Bar{u} \langle x \rangle^{2(m+1)}dx \Big) \\
        \|\langle x \rangle^{m+1}u\|_{L^2}\frac{d}{dt}\|\langle x \rangle^{m+1}u\|_{L^2} &\leq 2(m+1)\int_{-\infty}^{\infty} \langle x \rangle^{2m+1} |\partial_x u| |u| \, dx 
        \leq 2(m+1) \|\langle x \rangle^{m+1}u\|_{L^2} \|\langle x \rangle^m \partial_x u \|_{L^2}.
    \end{align*}
Applying the assumed inequality \eqref{E:m} with $u$ replaced by $\partial_x u$, then integrating in time, we obtain
    \begin{align*}
        \|\langle x \rangle^{m+1}u\|_{L^2} & \leq \|\langle x \rangle^{m+1}\psi\|_{L^2} + 2(m+1)\int_0^t\Big(C(1+s)^m \sum_{j=0}^m\|\langle x \rangle^{m-j}\partial_x^{j+1} \psi\|_{L^2}\Big) \,ds \\
        &\leq \|\langle x \rangle^{m+1} \psi\|_{L^2} + C (1+t)^{m+1} \sum_{j=0}^m\|\langle x \rangle^{m-j}\partial_x^{j+1} \psi\|_{L^2}, 
    \end{align*}
where $C$ changes from line to line, incorporating all constants.
Re-indexing $j+1$ with $j$, we deduce 
\begin{align*}
        \|\langle x \rangle^{m+1}u\|_{L^2} 
        &\leq \|\langle x \rangle^{m+1}\psi\|_{L^2} + C (1+t)^{m+1} \sum_{j=1}^{m+1}\|\langle x \rangle^{m-j + 1}\partial_x^{j} \psi\|_{L^2} 
        \leq C (1+t)^{m+1} \sum_{j=0}^{m+1}\|\langle x \rangle^{m +1-j}\partial_x^{j} \psi\|_{L^2},
    \end{align*}
completing the proof. 
\end{proof}

{
For clarity of the exposition we stated Lemma \ref{Int_weight} for Schwartz data, but this can be stated for a larger class of weighted Sobolev spaces. In particular, for our space $\mathcal{X}$ this brings us to the following remark.}
\begin{remark}
Recalling the definition of $J^m$ and the Sobolev space of order $m$ from \S \ref{S:notation}, using interpolation we may rewrite the $L^2$ term in the sum in \eqref{Rmk1}, eliminating the decreasing weights $\langle x \rangle^{m-j}$ except for $j=0$ (and thus, in part III of the space $\mathcal X$) as follows 
\begin{equation}\label{E:12}
\|\langle x \rangle^{m}e^{it\partial_x^2}f\|_{L^2} \leq C \langle t \rangle^m\big( \|\langle x\rangle^m f\|_{L^2} + \|J^{m}f\|_{L^2} \big).
\end{equation}
\end{remark}
We note that the expression in \eqref{E:12} works not only for an integer $m$ but for any real $b>0$, though we do not use it in this paper. For a complete proof with an arbitrary weight we refer a reader to 
\cite[Lemma 2]{NahasPonce2009} and \cite[Lemma 2.9]{ARR2021}. 
\smallskip

The main focus of this section is the proof of the following {\it linear} estimates. 
\begin{theorem}\label{linearEst}
Let the space $\mathcal{X}$ be defined by \eqref{Xspace} and \eqref{Xnorm} and let $v(x,t) = e^{it\partial_x^2}v_0$ with  
$v_0 \in \mathcal{X}$. 
Then for $t > 0$, there exists a constant $C > 0$ such that
\begin{equation}{\label{Prop1Pt1}}
\begin{aligned}
\big\|e^{it\partial_x^2}v_0 \big\|_{\mathcal{X}} \leq C{\langle t \rangle^{n+1}}\norm{v_0}_{\mathcal{X}}.
    \end{aligned}
\end{equation}
Furthermore,
\begin{equation}{\label{Prop1p2}}
\begin{aligned}
\big\|\langle x \rangle^n (e^{it\partial_x^2} v_0 - v_0) \big\|_{L^{\infty}} \leq C|t|\langle t \rangle^{n}\norm{v_0}_{\mathcal{X}}.
    \end{aligned}
\end{equation}
\end{theorem}
\begin{proof}
Set $v(t) = e^{it\partial_{x}^{2}} v_0$. Then using \eqref{E:linFTC}, we estimate each term of the $\mathcal{X}$ space, starting with part I of $\mathcal{X}$ the $L^\infty$ estimates. Using the Sobolev embedding $H^1(\mathbb R) \hookrightarrow L^\infty(\mathbb R)$, we obtain
\begin{equation}\label{E:v0}
        \norm{\langle x \rangle^{n} v(t)}_{L^{\infty}} 
        \leq \norm{\langle x \rangle^{n} v_0}_{L^{\infty}} + \int_{0}^{t} \norm{\langle x \rangle^{n} \partial_x^2 v(s)}_{L^{\infty}} \, ds
        \leq \norm{v_0}_{\mathcal{X}} + C\sum_{\beta = 2}^3\int_{0}^{t} \norm{\langle x \rangle^{n} \partial_x^\beta v(s)}_{L^2} \, ds.
\end{equation}
Recall, by the definition of space $\mathcal{X}$ in part II, we require $\langle x \rangle^{n} \partial_x^\beta v(s) \in L^2$ with the derivatives of order $\beta$ up to some value $r$. The above computation shows that $r$ must be at least $3$ (i.e., $r \geq 3$).  
Before continuing with the proof we make the following remark. 
\begin{remark}\label{UpBoundJ}
    For the following estimates, notice that $J^\beta$ is defined via the Fourier multiplier $\langle \xi \rangle^s$, then for any two powers $s$ and $l$ such that $0< s < l$, the  $\langle \xi \rangle^s \leq \langle \xi \rangle^l$ holds trivially. Thus, we have the relation for $J$ as
    $$\|J^{s} f\|_{L^2} \leq C\|J^{l} f\|_{L^2}.$$
\end{remark}
\begin{remark} We also use the following fact for the linear estimates. For $f \in \mathcal{X}$ and $r,n$ as in \eqref{XspaceCond1}, let $n_0 = \lceil n \rceil$, then
\begin{equation*}
    \|J^n \partial_x^\beta f \|_{L^2} \leq \|\partial_x^\beta f\|_{L^2} + \|\partial_x^{n_0 + \beta}f\|_{L^2} \leq C\|f\|_{\mathcal{X}},
\end{equation*}
where we have used the fact that for $1 \leq \beta \leq r$ we have $1 \leq \langle x \rangle^n$. Thus,
\begin{equation*}
    \|\partial_x^\beta f \|_{L^2} \leq  \|\langle x \rangle^n \partial_x^\beta f \|_{L^2} \leq \| f \|_{\mathcal X}.
\end{equation*}
\end{remark}

Applying \eqref{E:12} with $f = \partial^\beta v_0$ (note $\beta \geq 1$), we have 
\begin{equation*}
    \begin{aligned}
        \norm{\langle x \rangle ^{n}\partial_x^\beta v_0}_{L^2} \leq C \langle t \rangle^n\big( \|J^{n}\partial_x^\beta v_0\|_{L^2}+\|\langle x\rangle^n \partial_x^\beta v_0\|_{L^2} \big) 
        \leq C \langle t \rangle^n\big( \|J^{n}\partial_x^\beta v_0\|_{L^2}+\|\langle x\rangle^n \partial_x^\beta v_0\|_{L^2} \big).
    \end{aligned}
\end{equation*}
From the above estimate, since $\beta$ goes up to  $r \geq 3$, we obtain the condition that $ n + r \leq M$ and $n \leq r$ since $M \leq 2r$ (as we will see in the nonlinear estimates), and continue with bounding \eqref{E:v0} as
\begin{align}
        \norm{\langle x \rangle^{n} v(t)}_{L^{\infty}}
        &\leq \norm{v_0}_{\mathcal{X}} + C \langle t \rangle^n\sum_{\beta = 2}^3\int_{0}^{t} \big( \|J^n \partial_x^\beta v_0\|_{L^2}+\|\langle x\rangle^n \partial_x^\beta v_0\|_{L^2} \big) \, ds \label{E:v1a} \\ 
        &\leq \norm{v_0}_{\mathcal{X}} + C|t|\langle t \rangle^{n} \|v_0\|_{\mathcal X}  +C |t|\langle t \rangle^{n} \sum_{\beta = 2}^3 \big(\|\langle x\rangle^n \partial_x^\beta v_0\|_{L^2} \big)
        \leq C\langle t \rangle^{n+1} \norm{v_0}_{\mathcal{X}}.\label{E:v1b}
\end{align}

Next, we would like to estimate the part II of the space $\mathcal X$, the weighted terms with derivatives in $L^2$. Let $1 \leq \beta \leq r$. Applying directly \eqref{E:12}, we have 
\begin{equation}\label{E:v2}
\norm{\langle x \rangle^n \partial_x^\beta v(t)}_{L^2} \leq C \langle t \rangle^n\big( \|J^{n}\partial_x^\beta v_0\|_{L^2}+\|\langle x\rangle^n \partial_x^\beta v_0 \|_{L^2} \big)
\leq C\langle t \rangle^{n+1}\norm{v_0}_{\mathcal{X}}.
\end{equation}

Now, for part III of $\mathcal{X}$, using the isometry in $L^2$, Plancherel's formula and the fact that $1 \leq \langle t \rangle^{n+1}$, for $r + 1 \leq \beta \leq M$, we deduce 
\begin{equation}\label{E:v3}
        \norm{\partial_x^\beta v(t)}_{L^2} \leq C\langle t \rangle^{n+1}\norm{v_0}_{\mathcal{X}}.
\end{equation}
Combining the estimates \eqref{E:v1b}, \eqref{E:v2},\eqref{E:v3}, we obtain \eqref{Prop1Pt1}.

We now prove the estimate \eqref{Prop1p2}. Using \eqref{E:linFTC}, we rewrite the difference and apply a similar argument as earlier in the proof, to obtain
\begin{equation*}
        \norm{\langle x \rangle^n(v(t) - v_0)}_{L^\infty} \leq \int_0^t\norm{\langle x \rangle^n \partial^2_xv(t)}_{L^\infty}\,ds \leq C|t|\langle t \rangle^n\norm{v_0}_{\mathcal{X}},
\end{equation*}
completing the proof of Theorem \ref{linearEst}.
\end{proof}
\subsection{Weighted nonlinear estimates}
We start with estimating a single nonlinearity and then its difference in the space $\mathcal X$ (defined by \eqref{Xspace} and \eqref{Xnorm}), recalling the 
non-vanishing condition \eqref{inf}. 
\begin{theorem}\label{NonlinearEst} 
Let $u \in \mathcal{X}$ be such that for some $\eta >0$ the non-vanishing condition \eqref{inf} holds. 

Then for any $\alpha>0$, $|u|^{\alpha}u\in\mathcal{X}$. Moreover, there exists a constant $C > 0$ such that
\begin{equation}\label{Thm3.1Pt1}
    \norm{|u|^{\alpha}u}_{\mathcal{X}} \leq C\left(1+ \eta\norm{u}_{\mathcal{X}}\right)^{2M}\norm{u}_{\mathcal{X}}^{\alpha + 1}.
\end{equation}
Furthermore, for $u, v \in \mathcal{X}$ both satisfying \eqref{inf} for some $\eta>0$, there exists a constant $C>0$ such that
\begin{equation}\label{Thm3.1Pt2}
    \norm{|u|^{\alpha}u - |v|^{\alpha}v}_{\mathcal{X}} \leq C(1+\eta)^{4M +\alpha+ 2} \left(1 + \norm{u}_\mathcal{X} + \norm{v}_\mathcal{X}\right)^{4M + \alpha + 2}\norm{u-v}_\mathcal{X}.
\end{equation}
\end{theorem}

{To prove Theorem \ref{NonlinearEst}, we must show that for each nonlinear term $\mathcal{N}_\alpha(u)  = |u|^\alpha u $, we have $\|\mathcal{N}_\alpha(\cdot)\|_\mathcal{X} < \infty$. By the definition of our space, \eqref{Xspace}, it is sufficient to split the estimates into three parts: \textbf{I.} $\|\langle x \rangle^n\mathcal{N}_\alpha(\cdot)\|_{L^\infty} < \infty$, \textbf{II.} $\|\langle x \rangle^n\partial_x^\beta\mathcal{N}_\alpha(\cdot)\|_{L^2} < \infty$ (for $\beta \in [1,r]$), and \textbf{III.} $\|\partial_x^\beta\mathcal{N}_\alpha(\cdot)\|_{L^2} < \infty$ (for $\beta \in [r+1,M]$)}.

\begin{proof}
Let $u \in \mathcal X$ and let $\eta>0$ be such that \eqref{inf} holds for a given $u$. We estimate $|u|^\alpha u$ in each of the three parts (I, II, III) of the space $\mathcal X$ \eqref{Xspace}.

{\bf I.} \textbf{$L^\infty$ estimates}. Since no derivatives are involved, with trivial re-arrangement, we obtain
\begin{equation}\label{Linfty_estimate}
\norm{\langle x \rangle^n (|u|^{\alpha}u)}_{L^{\infty}} 
\leq 
\norm{\langle x \rangle^{-n \alpha}}_{L^{\infty}}
\norm{\langle x \rangle^n |u|}^{\alpha}_{L^{\infty}} \norm{\langle x \rangle^n u}_{L^{\infty}} 
\leq
C \norm{u}^{\alpha + 1}_{\mathcal{X}}.
\end{equation}

{\bf II.} \textbf{Weighted $L^2$ estimates ($1\leq \beta \leq r$)}.  
In this part we need to show that for a given range of $\beta$, we have $\norm{\langle x \rangle^n \partial_x^{\beta} (|u|^{\alpha}u)}_{L^{2}} < \infty$. 
The idea is to differentiate the nonlinearity similar to a `product' rule:
\begin{equation}\label{E:PR}
\partial_x^\beta (|u|^{\alpha} u) \sim
|u|^{\alpha}\partial_x^\beta u +
\sum_\rho \partial_x^{\beta-\rho} (|u|^{\alpha}) \partial_x^\rho u.
\end{equation}
To make this rigorous, we write
{\small
\begin{align}\label{E-IIa}
\big\| \langle x \rangle^n \partial_x^\beta (|u|^{\alpha} u) \big\|_{L^{2}} 
&\leq \underbrace{\norm{\langle x \rangle^n |u|^{\alpha} \partial_x^{\beta} u }_{L^2}}_{A} + \underbrace{\Big\| {\langle x \rangle^n \sum_{\substack{\gamma + \rho = \beta \\ \gamma \neq 0}}\sum_{k = 1}^{\gamma}\sum_{\substack{l_1+...+l_k=\gamma, \\ l_j \geq 1}} \tilde C_{l_1...l_k}\Big(|u|^{\alpha - 2k}\prod_{j=1}^{k}\partial_x^{l_j}(|u|^2) \Big)\partial_x^\rho u} \Big\|_{L^2}}_{B},
\end{align}
}
and estimate each part separately. 

$\bullet$ \underline{Estimate of A}: We put the derivative term into $L^2$ and pull out the rest into the $L^\infty$ norm,
\begin{equation}\label{E:II_finalA}
    \norm{\langle x \rangle^n |u|^{\alpha} \partial_x^{\beta} u}_{L^2} \leq C\|\langle x \rangle^{-n\alpha}\|_{L^\infty}\norm{\langle x \rangle^n \partial_x^{\beta}u}_{L^2} \norm{|\langle x \rangle^nu|^{\alpha}}_{L^{\infty}}\leq C\norm{u}_{\mathcal{X}}^{\alpha + 1}.
\end{equation}

$\bullet$ \underline{Estimate of B}:
The term $B$ in \eqref{E-IIa} can be expanded further, using the derivative of $|u|^2$ \eqref{E:Sum_of_products}, leaving the second term in \eqref{genchain}. Applying a triangle inequality over all sums, we obtain
\begin{equation}\label{expanded_B}
    B \leq  
        \sum_{\substack{\gamma + \rho = \beta \\ \gamma \neq 0}} \sum_{k=1}^\gamma\sum_{\substack{\ell_{1}+...+\ell_{k} = \gamma \\ \ell_{j} \geq 1}}  \sum_{\substack{m_{1,j} + m_{2,j} = \ell_j\\ j = 1,...,k}}C \underbrace{\|
        \langle x \rangle^n
        |u|^{\alpha - 2k}\prod_{j=1}^k \big(\partial_x^{m_{1,j}}u \partial_x^{m_{2,j}}\bar{u} \big) \partial_{x}^{\rho}u\|_{L^2}}_{\mathcal{B}}.
\end{equation}
Thus, we only need to estimate the term in the $L^2$ norm of \eqref{expanded_B}. {Here, we are using the non-vanishing condition \eqref{inf} to control terms such as $|u|^{-k}, \;k>0.$ Note that in the case for Strichartz estimates, we did not have this flexibility}. To be precise, for any $x \in \mathbb R$, $\inf |\langle x \rangle^n u| > \frac{1}{\eta}$, thus, $|u|^{-2k} \leq \eta^{2k}\langle x \rangle^{2kn}$, we get 
\begin{align}
\mathcal{B} 
    &\leq \eta^{2k} \|\langle x \rangle^{n} |u|^{\alpha} \langle x \rangle^{2kn}\prod_{j=1}^k\left(\partial_x^{m_{1,j}} u\partial_x^{m_{2,j}}\bar{u}\right)\partial_{x}^{\rho}u\|_{L^2} \notag\\
    &\leq \eta^{2k} \||u|^{\alpha} \prod_{j=1}^k
    \big( (\langle x \rangle^{n}\partial_x^{m_{1,j}}u) 
    (\langle x\rangle^{n}\partial_x^{m_{2,j}}\bar{u} )\big)
    (\langle x \rangle^{n}\partial_{x}^{\rho}u) \|_{L^2} \notag \\
    &\leq \eta^{2k}\|\langle x \rangle^{-n\alpha}\|_{L^\infty} \|\langle x \rangle ^n u\|_{L^\infty}^{\alpha} \underbrace{\|\prod_{j=1}^k \left( (\langle x \rangle^{n}\partial_x^{m_{1,j}}u) (\langle x \rangle^{n} \partial_x^{m_{2,j}}\bar{u} ) \right) (\langle x\rangle^{n} \partial_{x}^{\rho}u)\|_{L^2}}_{(*)}.\label{IIa}
\end{align}

 In $(*)$, there are $2k + 1$ many terms of type $\langle x\rangle^{n} \partial_{x}^{\gamma}u$ to estimate: only one term will be given the $L^2$ norm with the rest of the terms supped out with the $L^\infty$ norm. We choose the derivative of highest order to estimate in the $L^2$ norm. Notice that any $m_{i,j}$ or $\rho$ could be the highest order of the derivative (label as $\beta_{max}$). Without loss of generality, choose $j^*$ to be the highest index such that $m_{1,j^*} = \beta_{max}$. By our choice of $\beta$,  $1 \leq \beta_{max} \leq r$, and thus, all other $m_{i,j}, \rho < \frac{r}{2}$ (this can occur if we only have 2 terms with derivatives where $\beta_{max} = \frac{r}{2}$, since $\frac{r}{2} + \frac{r}{2} = r$). Then
 \begin{align}
     \mathcal{B} &\leq \eta^{2k}\|\langle x \rangle^{-n\alpha}\|_{L^\infty} \|\langle x \rangle ^n u\|_{L^\infty}^{\alpha}\|\langle x \rangle^n \partial^{m_{1,1}}u\|_{L^\infty}\|\langle x \rangle^n \partial^{m_{2,1}}u\|_{L^\infty}\cdot \ldots \notag \\
     &\hspace{0.2in}\ldots \cdot\|\langle x \rangle^n \partial^{\beta_{max}}u\|_{L^2}\|\langle x \rangle^n \partial^{m_{2,j^*}}u\|_{L^\infty}\cdot \ldots \cdot \|\langle x \rangle^n \partial^{m_{1,k}}u\|_{L^\infty} \|\langle x \rangle^n \partial^{m_{2,k}}u\|_{L^\infty}\|\langle x \rangle^n \partial^{\rho}u\|_{L^\infty}. \label{IIb}
 \end{align}

We previously established that all $m_{i,j}, \rho \leq \frac{r}{2}$ (except, of course, $m_{1,j^*}$, since that is $\beta_{max}$). In the case where any $m_{i,j}, \rho = 0$, by the definition of space $\mathcal{X}$, they are bounded in $L^\infty$. Otherwise, we apply a Sobolev embedding,  $1 \leq m_{i,j}+1, \rho+1 \leq \frac{r+2}{2} < r$ if $r > 2$ (which is guaranteed). Hence, 
\begin{align*}
     \mathcal{B} &\leq C \eta^{2k}\|u\|_{\mathcal{X}}^{\alpha}\|\langle x \rangle^n \partial^{m_{1,1}}u\|_{H^1}\|\langle x \rangle^n \partial^{m_{2,1}}u\|_{H^1}\cdot \ldots \\
     &\hspace{0.2in}\ldots \cdot\|u\|_{\mathcal{X}}\|\langle x \rangle^n \partial^{m_{2,j^*}}u\|_{H^1}\cdot \ldots \cdot \|\langle x \rangle^n \partial^{m_{1,k}}u\|_{H^1} \|\langle x \rangle^n \partial^{m_{2,k}}u\|_{H^1}\|\langle x \rangle^n \partial^{\rho}u\|_{H^1}\\
     &\leq C \eta^{2k}\|u\|_{\mathcal{X}}^{\alpha + 2k +1}.
 \end{align*}
Therefore,
\begin{equation}\label{E:II_finalB}
     B \leq C\|u\|_{\mathcal{X}}^{\alpha +1}\sum_{k=1}^\beta \eta^{2k}\|u\|_{\mathcal{X}}^{2k}.
 \end{equation}
Since $\beta < M$, our estimates of $A$ and $B$ together yield
\begin{equation}\label{II_finalbound}
     \|\langle x \rangle^n \partial^\beta(|u|^\alpha u)\|_{L^2} \leq C(1+\eta\|u\|_{\mathcal{X}})^{2M}\|u\|^{\alpha+1}_{\mathcal{X}}.
 \end{equation}

{\bf III.} \textbf{Unweighted $L^2$ estimates ($r+1 \leq \beta \leq M\leq 2r$)}. 
For the given range of $\beta$, and recalling \eqref{E-IIa} from part II but without the weight $\langle x \rangle^n$, we write
\begin{equation*}
\begin{aligned}
            \norm{ \partial_x^{\beta} (|u|^{\alpha}u)}_{L^{2}} &\leq \underbrace{C \norm{ |u|^{\alpha} \partial_x^{\beta} u}_{L^2}}_{\mathrm{III}a} +  \underbrace{C\left(\sum\right)^4 \Big\| \Big(|u|^{\alpha - 2k}  \prod^k_{j=1} \partial_x^{m_{1,j}} u \partial_x^{m_{2,j}} \overline{u}\Big)\partial^\rho u\Big\|_{L^2}}_{\mathrm{III}b}. 
        \end{aligned}
    \end{equation*}
For the term $\mathrm{III}a$, putting the derivative into $L^2$ norm and the rest into $L^\infty$, we obtain 
\begin{equation}\label{IIIa}
    \mathrm{III}a = C\norm{|u|^{\alpha} \partial_x^{\beta} u}_{L^2} \leq C\norm{\partial_x^{\beta}u}_{L^2} 
     \norm{\langle x \rangle^n u}_{L^\infty}^\alpha 
    \leq C\norm{ u}_{\mathcal{X}}^{\alpha+1}.
\end{equation}
For the second term, $\mathrm{III}b$, we estimate the term in the $L^2$ norm, labeling it as
\begin{equation*}
    \mathcal{B} \defeq \Big\| \Big(|u|^{\alpha - 2k}  \prod^k_{j=1} \partial_x^{m_{1,j}} u \partial_x^{m_{2,j}} \overline{u}\Big)\partial^\rho u\Big\|_{L^2}. 
\end{equation*}
Then, recalling the bound $|u|^{-2k} \leq \eta^{2k}\langle x \rangle^{2kn}$, we deduce
\begin{equation}\label{B-IIIa}
\mathcal {B}
\leq \eta^{2k} \|\langle x \rangle^{-n\alpha} |\langle x \rangle^n u|^{\alpha} \langle x \rangle^{2kn}\underbrace{\prod_{j=0}^k\left(\partial_x^{m_{1,j}}u \partial_x^{m_{2,j}}\bar{u}\right)\partial_{x}^{\rho}u}_{(*)}\|_{L^2}.
\end{equation}

We have two possibilities in $(*)$ of \eqref{B-IIIa}, either: (1) all derivatives $m_{i,j}, \;\rho \leq r$ or (2)
there is one derivative of order greater than $r$ and the rest are smaller than $r$. Notice that we can never have 2 derivatives of order $r + 1$ or larger, since the total order is $\beta \leq 2r$.

{
\begin{remark}
This is where we obtain the restriction that the maximum regularity achievable in this method is at most $M \leq 2r$, as defined in space $\mathcal{X}$.
\end{remark}}

$\bullet$ {\bf \underline{sub-case $(1)$}}. In this case all $m_{i,j}, \; \rho \leq r$ in $(*)$. Without loss of generality, choose $j^*$ to be the index such that $m_{i,j^*} = \beta_{max} < r$, then all other $m_{i,j}, \; \rho < r$. In this scenario, we have to add an extra weight (by multiplying and dividing $\langle x \rangle^n$) and then distribute the other $\langle x \rangle^{2kn}$ weights to all terms and then apply the H\"older inequality, to obtain 
\begin{align}
     \mathcal {B} &\leq \eta^{2k}\|\langle x \rangle^{-n\alpha- n}\|_{L^\infty} \|\langle x \rangle ^n u\|_{L^\infty}^{\alpha}\|\langle x \rangle^n \partial^{m_{1,1}}u\|_{L^\infty}\|\langle x \rangle^n \partial^{m_{2,1}}u\|_{L^\infty}\cdot \ldots \notag\\
     &\hspace{0.2in}\ldots \cdot\| \langle x \rangle^n\partial^{\beta_{max}}u\|_{L^2}\|\langle x \rangle^n \partial^{m_{2,j^*}}u\|_{L^\infty}\cdot \ldots \cdot \|\langle x \rangle^n \partial^{m_{1,k}}u\|_{L^\infty} \|\langle x \rangle^n \partial^{m_{2,k}}u\|_{L^\infty}\|\langle x \rangle^n \partial^{\rho}u\|_{L^\infty}\label{IIIb1}.
 \end{align}
We can  then apply the Sobolev embedding to all $L^\infty$ terms (that have derivatives of order larger than 0).
Note that there is a specific scenario here that we must be careful of. In the chance that $\beta_{max} = r$, then there may be at most one other derivative (let's call it $\partial_x^\rho u$) such that $\rho = r$ (all other $m_{i,j} = 0$ here). In that case, we do not multiply and divide by an extra weight, and instead assign the $L^\infty$ norm to $\partial^{\beta_{max}}u$ without a weight, and the $L^2$ norm is assigned to $\partial^{\rho}u$ with the weight. We proceed as usual by applying the Sobolev embedding to all appropriate terms with $L^\infty$ norms 
\begin{align}
     \mathcal {B} &\leq \eta^{2k}\|\langle x \rangle^{-n\alpha}\|_{L^\infty} \|\langle x \rangle ^n u\|_{L^\infty}^{\alpha}\|\langle x \rangle^n \partial^{m_{1,1}}u\|_{L^\infty}\|\langle x \rangle^n \partial^{m_{2,1}}u\|_{L^\infty}\cdot \ldots \notag\\
     &\hspace{0.2in}\ldots \cdot\| \partial^{\beta_{max}}u\|_{L^\infty}\|\langle x \rangle^n \partial^{m_{2,j^*}}u\|_{L^\infty}\cdot \ldots \cdot \|\langle x \rangle^n \partial^{m_{1,k}}u\|_{L^\infty} \|\langle x \rangle^n \partial^{m_{2,k}}u\|_{L^\infty}\|\langle x \rangle^n \partial^{\rho}u\|_{L^2} \notag \\ 
     &\leq C\eta^{2k} \| u\|_{\mathcal{X}}^{\alpha}\|\langle x \rangle^n \partial^{m_{1,1}}u\|_{L^\infty}\|\langle x \rangle^n \partial^{m_{2,1}}u\|_{L^\infty}\cdot \ldots \notag\\
     &\hspace{0.2in}\ldots \cdot\| \partial^{\beta_{max}}u\|_{H^1}\|\langle x \rangle^n \partial^{m_{2,j^*}}u\|_{L^\infty}\cdot \ldots \cdot \|\langle x \rangle^n \partial^{m_{1,k}}u\|_{L^\infty} \|\langle x \rangle^n \partial^{m_{2,k}}u\|_{L^\infty}\|\langle x \rangle^n \partial^{\rho}u\|_{L^2}
     \label{IIIb2}.
 \end{align}
Counting the number of norms, we obtain 
\begin{equation*}
     \mathcal {B} \leq C \eta^{2k}\|u\|_{\mathcal{X}}^{\alpha + 2k +1}.
 \end{equation*}

$\bullet$ {\bf \underline{sub-case $(2)$}}. In this case, we only have one $m_{i,j}$ or $\rho$ of order $\geq r+1$.  In fact, only one $m_{i,j}$ or $\rho$ can be larger than $r+1$ at a time. Without loss of generality, let $j^*$ be the index where $m_{1,j^*} = \beta_{max}$. Since $\beta_{max} \geq r+1$, then for all $i,j$ such that $j\neq j^*$, $m_{i,j},\rho \leq r-1$. Crucially, this ensures that all these terms require a weight, except the term $\partial_x^{\beta_{max}}u$. Thus,
\begin{align}
     \mathcal {B} &\leq \eta^{2k}\|\langle x \rangle^{-n\alpha}\|_{L^\infty} \|\langle x \rangle ^n u\|_{L^\infty}^{\alpha}\|\langle x \rangle^n \partial^{m_{1,1}}u\|_{L^\infty}\|\langle x \rangle^n \partial^{m_{2,1}}u\|_{L^\infty}\cdot \ldots \notag\\
     &\hspace{0.2in}\ldots \cdot\| \partial^{\beta_{max}}u\|_{L^2}\|\langle x \rangle^n \partial^{m_{2,j^*}}u\|_{L^\infty}\cdot \ldots \cdot \|\langle x \rangle^n \partial^{m_{1,k}}u\|_{L^\infty} \|\langle x \rangle^n \partial^{m_{2,k}}u\|_{L^\infty}\|\langle x \rangle^n \partial^{\rho}u\|_{L^\infty}\label{IIIb3},
\end{align}
which implies
\begin{equation*}
     \mathcal {B} \leq C \eta^{2k}\|u\|_{\mathcal{X}}^{\alpha + 2k +1}.
\end{equation*}
Summing up, the IIIb term is bounded similar to \eqref{E:II_finalB}. Combining this with \eqref{IIIa} for IIIa, we obtain the same upper bound as in \eqref{II_finalbound},
completing the proof of \eqref{Thm3.1Pt1} in Theorem \ref{NonlinearEst}.
\smallskip

\noindent{\it Difference estimate.} 
We now proceed to the difference $|u|^{\alpha}u - |v|^{\alpha}v$ to prove \eqref{Thm3.1Pt2}.

{\bf I. }\textbf{$L^\infty$ estimates}. For this estimate, we directly apply \eqref{meanvalueineq} to get
    \begin{equation}\label{diffLinfty}
        \begin{aligned}
            \norm{\langle x \rangle^n (|u|^{\alpha}u - |v|^{\alpha}v)}_{L^{\infty}}
            &\leq C\norm{\langle x \rangle^{- \alpha n}}_{L^\infty}(\norm{\langle x \rangle^n u}^{\alpha}_{L^{\infty}} + \norm{\langle x \rangle^n v}^{\alpha}_{L^{\infty}})\norm{\langle x \rangle^n (u-v)}_{L^{\infty}}\\
            &\leq C(\norm{u}^{\alpha}_{\mathcal{X}} + \norm{v}^{\alpha}_{\mathcal{X}}) \norm{u-v}_{\mathcal{X}}.
        \end{aligned}
    \end{equation}

{\bf II. }\textbf{Weighted $L^2$ estimates.}
We want to show that $\langle x \rangle^n \partial_x^{\beta} (|u|^{\alpha}u - |v|^{\alpha}v) \in L^2(\mathbb{R})$ for $1 \leq \beta \leq r$.
Using Lemma \ref{L:Difference_derivative}, we rewrite the difference as 
    \begin{equation*}
        \begin{aligned}
            \left| \partial_x^\beta(|u|^\alpha u) - \partial_x^\beta(|v|^\alpha v )\right| \leq A + B,
        \end{aligned}
    \end{equation*}
    with $A$ and $B$ as in \eqref{E:A} and \eqref{E:B}, respectively.

To estimate $A$, we first utilize \eqref{derivdiffinf}, to get
    \begin{equation}\label{L2_pt1_Asplit}
        \begin{aligned}
            A 
            &\leq \underbrace{C_{0,\beta}\Omega\norm{\langle x \rangle^{2n - \alpha n}|u-v|\partial_x^\beta v}_{L^2}}_{A_1} + \underbrace{C_{0,\beta}\norm{\langle x \rangle^n \left(\partial_x^\beta u- \partial_x^\beta v\right)|u|^{\alpha}}_{L^2}}_{A_2},\\
        \end{aligned}
    \end{equation}
where the constant $\Omega$ is defined by,
\begin{equation*}
    (\eta^{1-\alpha}+ C(|\langle x \rangle^n u|^{|\alpha-1|} + |\langle x \rangle^n v|^{|\alpha-1|})) \leq C(\eta^{1-\alpha}+ (\|u\|^{|\alpha-1|}_{\mathcal{X}} + \|v\|^{|\alpha-1|}_{\mathcal{X}})) \defeq \Omega.
\end{equation*}
    
Thus,
    \begin{equation*}
        \begin{aligned}
            A_1 
            &\leq C \Omega\norm{\langle x \rangle^{-n\alpha}}_{L^{\infty}} \norm{\langle x \rangle^n(u-v)}_{L^\infty} \norm{\langle x \rangle^n \partial_x^\beta v}_{L^2}\\
            &\leq C(\eta^{1-\alpha}+ (\|u\|^{|\alpha-1|}_{\mathcal{X}} + \|v\|^{|\alpha-1|}_{\mathcal{X}}))\norm{u-v}_{\mathcal{X}} \norm{v}_{\mathcal{X}},
        \end{aligned}
    \end{equation*}
    and
    \begin{equation*}
        \begin{aligned}
            A_2
            &\leq C\norm{\langle x \rangle^n \partial_x^\beta(u-v)}_{L^2} \norm{\langle x \rangle^n v}_{L^{\infty}}^{\alpha} \norm{\langle x \rangle^{-n\alpha}}_{L^\infty}\\
            &\leq C\norm{u-v}_\mathcal{X} \norm{v}_\mathcal{X}^{\alpha}.
        \end{aligned}
    \end{equation*}
    Finally,
    \begin{equation}\label{L2_pt1_A}
        A \leq C(\eta^{1-\alpha}+ (\|u\|^{|\alpha-1|}_{\mathcal{X}} + \|v\|^{|\alpha-1|}_{\mathcal{X}}))\norm{u-v}_{\mathcal{X}} \norm{v}_{\mathcal{X}} + C\norm{u-v}_\mathcal{X} \norm{v}_\mathcal{X}^{\alpha}.
    \end{equation}
    
For $B$ we apply the grouping as in Corollary \ref{Cor:B_splitting} with $B_1$, $B_2$ and $B_3$ as in \eqref{E:B1}, \eqref{E:B2}, and \eqref{E:B3_unsplit}, respectively. 
It is important to specify that throughout these estimates, the Sobolev embedding $H^1(\R) \hookrightarrow L^\infty(\R)$ will be used as needed. In fact, the protocol for applying these estimates will be the same as in part {\bf II.} for the regular nonlinear estimates. This will be apparent especially in what follows.

To begin the estimate for $B_1$, we  split it into two pieces by using the difference estimate \eqref{difference2k} as      
\begin{equation*}
            \begin{aligned}
                &\|\langle x \rangle^n B_1\|_{L^2}\leq I + II,
            \end{aligned}
        \end{equation*}
        where
        \begin{align*}
            I &= C\eta^{2+4k}\Big\|\langle x \rangle^n \langle x \rangle^{2n(1+2k)} (|u|^{2k} + |v|^{2k}) |u-v| |u|^{\alpha+1} \partial_x^\rho u\prod_{j=1}^k \partial_x^{m_{1,j}}u \partial_x^{m_{2,j}}\bar{u}\Big\|_{L^2},\\
            II &= C\eta^{1+2k}\Big\|\langle x \rangle^n \langle x \rangle^{n(1+2k)}(|u|^\alpha + |v|^\alpha) |u-v| \partial_x^\rho u \prod_{j=1}^k \partial_x^{m_{1,j}}u \partial_x^{m_{2,j}}\bar{u}\Big\|_{L^2}.
        \end{align*}
Then by careful use of the H\"older inequality, we have
\begin{equation}\label{diff_IIa}
            \begin{aligned}
                I 
                &\leq C\eta^{2+4k}\norm{\langle x \rangle^{-n-n\alpha}}_{L^\infty}\left[ \norm{\langle x \rangle^n u}_{L^{\infty}}^{2k} + \norm{\langle x \rangle^n v}_{L^{\infty}}^{2k}\right]\norm{\langle x \rangle^n |u-v|}_{L^{\infty}} \norm{\langle x \rangle^n u}_{L^{\infty}}^{\alpha+1}...\\
                &\hspace{2.8in} ... \underbrace{\Big\| \prod_{j=1}^k \left(\langle x \rangle^n \partial_x^{m_{1,j}}u \langle x \rangle^n \partial_x^{m_{2,j}}\bar{u}\right) \langle x \rangle^n \partial_x^\rho u\Big\|_{L^{2}}}_{(*)},
            \end{aligned}
        \end{equation}        
        and
        \begin{equation}\label{diff_IIb}
            \begin{aligned}
                II 
                 &\leq C\eta^{1+2k}\norm{\langle x \rangle^{-\alpha n}}_{L^\infty} \left[\norm{\langle x \rangle^n u}^\alpha_{L^{\infty}} + \norm{\langle x \rangle^n v}^\alpha_{L^{\infty}} \right] \norm{\langle x \rangle^n |u-v|}_{L^{\infty}}...\\
                &\hspace{1.8in} ...\underbrace{\Big\| \prod_{j=1}^k \left(\langle x \rangle^n \partial_x^{m_{1,j}}u \langle x \rangle^n \partial_x^{m_{2,j}}\bar{u}\right) \langle x \rangle^n \partial_x^\rho u\Big\|_{L^{2}}}_{(*)}.
            \end{aligned}
        \end{equation}
Note that in \eqref{diff_IIa} and \eqref{diff_IIb}, the term $(*)$ is identical to that of \eqref{IIa} in part {\bf II}. Since  $1 \leq \beta \leq r$, we may use the same technique, to obtain 
\begin{equation*}
            I \leq C\eta^{2+4k}\left[\norm{u}_{\mathcal{X}}^{2k} + \norm{v}_{\mathcal{X}}^{2k}\right] \norm{u-v}_{\mathcal{X}} \norm{u}_{\mathcal{X}}^{2k+\alpha + 2},
        \end{equation*}
    and
        \begin{equation*}
            II \leq C\eta^{1+2k}\left[\norm{u}_{\mathcal{X}}^{\alpha} + \norm{v}_{\mathcal{X}}^{\alpha}\right] \norm{u-v}_{\mathcal{X}} \norm{u}_{\mathcal{X}}^{2k+1}.
        \end{equation*}
        
Finally, we have
\begin{equation}\label{L2_pt1_B1}
            \begin{aligned}
                \norm{\langle x \rangle^n B_1}_{L^2} \leq C_n\eta^{1+2k}\norm{u-v}_{\mathcal{X}} \norm{u}_{\mathcal{X}}^{2k+1}\left(\norm{u}_{\mathcal{X}}^{\alpha} + \norm{v}_{\mathcal{X}}^{\alpha} + \eta^{1+2k} \norm{u}_\mathcal{X}^{\alpha+1}\left[\norm{u}_{\mathcal{X}}^{2k} + \norm{v}_{\mathcal{X}}^{2k}\right]\right).
            \end{aligned}
        \end{equation}
The next term is more straightforward as we do not need to use a difference estimate, but note we still use the H\"older inequality and Sobolev embedding,
        \begin{equation*}
                \|\langle x \rangle^n B_2\|_{L^2}
                \leq \eta^{2k} \norm{\langle x \rangle^{-n(\alpha + 1)}}_{L^\infty} \norm{\langle x \rangle^n v}_{L^{\infty}}^\alpha \underbrace{\Big\|\prod_{j=1}^k \left( \langle x \rangle^n \partial_x^{m_{1,j}} u \langle x \rangle^n \partial_x^{m_{2,j}} \bar{u}\right) \langle x \rangle^n \partial_x^\rho (u-v)\Big\|_{L^{2}}}_{(*)}.
        \end{equation*}
Again, $(*)$ is analogous to that of \eqref{IIa}, but instead of $\partial^\rho u $, we have $\partial^\rho (u-v)$. This, however, can be treated the same, and thus,  we have
\begin{equation}\label{L2_pt1_B2}
            \|\langle x \rangle^n B_2\|_{L^2}\leq C \eta^{2k} \norm{v}_{\mathcal{X}}^\alpha \norm{(u-v)}_{\mathcal{X}} \norm{u}_{\mathcal{X}}^{2k}.
        \end{equation}
The last term $B_3$ is slightly involved, requiring extra grouping from Corollary \ref{Cor:B3_estimate}, the relation obtained from \eqref{inf} and more use of the H\"older inequality and Sobolev embedding. As established in Corollary \ref{Cor:B3_estimate}, it is sufficient to estimate $B_3'$ of \eqref{E:B3}, obtaining
\begin{align}
    \|\langle x \rangle^n B_3'\|_{L^2} 
    &\leq C \eta^{2k}\| \langle x \rangle ^{-\alpha n}\|_{L^\infty} \|\langle x \rangle ^n v\|_{L^\infty}^{\alpha} \Bigg\|\langle x \rangle ^n\partial_x^\rho v \Bigg[\sum_{i=1}^k \left( \prod_{j=1}^{i-1} \langle x \rangle ^n\partial^{m_{1,j}}v\langle x \rangle ^n\partial^{m_{2,j}}\bar v \right) \notag\\
    &\hspace{4cm}\times
    \langle x \rangle ^n(\partial^{m_{1,i}}u  - \partial^{m_{1,i}} v) \langle x \rangle ^n( \partial^{m_{2,i}}\bar u +\partial^{m_{2,i}}\bar v) \notag\\
    &\hspace{5.3cm}\times
    \left( \prod_{j=i+1}^{k} \langle x \rangle ^n\partial^{m_{1,j}}u \langle x \rangle ^n\partial^{m_{2,j}}\bar u \right)\Bigg] \Bigg\|_{L^2}.\label{B3_partestimate}
\end{align}
The last term (in the $L^2$ norm) of \eqref{B3_partestimate} has the same form as $(*)$ in \eqref{IIa},\eqref{diff_IIa} and \eqref{diff_IIb}. Since $1 \leq \beta \leq r$, we may again choose without loss of generality $j^*$ to be the index such that $m_{1,j^*}$ is the largest order of derivative. We assign to that term the $L^2$ norm with the weight, and the rest terms estimate in $L^\infty$, to obtain
\begin{align}\label{B3_fullestimatework}
    B_3' &\leq C\eta^{2k}\|v\|_{\mathcal{X}}^{\alpha+1}\left[\sum_{i=1}^k \|v\|_{\mathcal{X}}^{2(i-1)} \|u-v\|_{\mathcal{X}}(\|u\|_{\mathcal{X}}+\|v\|_{\mathcal{X}})\|u\|_{\mathcal{X}}^{2(k - i)}\right] \notag\\
    &\leq C\eta^{2k}(\|u\|_{\mathcal{X}}+\|v\|_{\mathcal{X}})^{2k+\alpha + 1}\|u-v\|_{\mathcal{X}}.
\end{align}
Notice that the final estimate has no dependence on the order of derivatives outside of a constant, so the final estimate for $B_3$, using \eqref{E:B3_unsplit}, is given by
\begin{equation}\label{L2_pt1_B3}
    \|\langle x \rangle^n B_3\|_{L^2} \leq C\eta^{2k}(\|u\|_{\mathcal{X}}+\|v\|_{\mathcal{X}})^{2k+\alpha + 1}\|u-v\|_{\mathcal{X}}.
\end{equation}
Collecting \eqref{L2_pt1_A}, \eqref{L2_pt1_B1}, \eqref{L2_pt1_B2}, and \eqref{L2_pt1_B3} with the following adjustments
\begin{equation*}
        \norm{u}_{\mathcal{X}}^\alpha \leq (1 + \norm{u}_{\mathcal{X}} + \norm{v}_{\mathcal{X}})^\alpha, \quad
        \eta^{2k} \leq (1 + \eta)^{2k},
    \end{equation*}
we obtain 
    \begin{equation}{\label{C2diffest}}
        \norm{\langle x \rangle^n(\partial_x^\beta|u|^\alpha u - \partial_x^\beta|v|^\alpha v)}_{L^2} \leq C_{n}(1+\eta)^{4M + \alpha + 2} \left(1 + \norm{u}_\mathcal{X} + \norm{v}_\mathcal{X}\right)^{4M + \alpha + 2}\norm{u-v}_\mathcal{X},
    \end{equation}
for an appropriately large constant $C>0$.
    
{\bf III. }\textbf{Unweighted $L^2$ estimates.} For $r+1 \leq \beta < M$, 
we again use Lemma \ref{L:Difference_derivative} and split
    \begin{equation*}
        \begin{aligned}
            \left| \partial_x^\beta(|u|^\alpha u) - \partial_x^\beta (|v|^\alpha v) \right| \leq A + B,
        \end{aligned}
    \end{equation*}
as in \eqref{E:nonlinear_diff_deriv} with $A$ and $B$ as in \eqref{E:A} and \eqref{E:B} respectively. In fact, the estimate for $A$ follows the same procedure as in \eqref{L2_pt1_Asplit}, but without weight. Thus, the final estimate for A is 
\begin{equation}\label{L2_pt_Aestimate}
        \norm{A}_{L^2} \leq C(\eta^{1-\alpha}+ (\|u\|^{|\alpha-1|}_{\mathcal{X}} + \|v\|^{|\alpha-1|}_{\mathcal{X}}))\norm{u-v}_{\mathcal{X}} \norm{v}_{\mathcal{X}} + \norm{u-v}_\mathcal{X} \norm{u}_\mathcal{X}^{\alpha}.
\end{equation}

The estimates for $B$ will mirror that of part {\bf III.} above in the single term nonlinear estimates. To emphasize this, we show the terms left to bound. We begin by splitting $B$ as in Corollary \ref{Cor:B_splitting} and using \eqref{difference2k} as in the previous estimates, we have the following:
\begin{equation*}
            \begin{aligned}
                &\norm{ B_1}_{L^2} 
                \leq I + II,
            \end{aligned}
\end{equation*}
with
\begin{align}\label{unweighted_I}
        I &\leq C\eta^{2+4k} \norm{\langle x \rangle^{ - n\alpha}}_{L^{\infty}} \left(\norm{\langle x \rangle^n u}^{2k}_{L^{\infty}} + \norm{\langle x \rangle^n v}^{2k}_{L^{\infty}}\right)...\notag\\
        &\qquad ...\norm{\langle x \rangle^n (u-v)}_{L^{\infty}} \norm{\langle x \rangle^n u}^{\alpha+1}_{L^{\infty}}\underbrace{\Big\|\langle x \rangle ^{2kn}\prod_{j=1}^k  \left(\partial_x^{m_{1,j}}u  \partial_x^{m_{2,j}}\bar{u}\right)\partial_x^\rho u\Big\|_{L^{2}}}_{(*)} 
\end{align}
and
\begin{align}\label{unweighted_II}
        II &\leq C\eta^{1+2k}\norm{\langle x \rangle^{-n\alpha}}_{L^{\infty}} \Big(\norm{\langle x \rangle^n u}^{\alpha}_{L^{\infty}} + \norm{\langle x \rangle^n v}^{\alpha}_{L^{\infty}}\Big) \norm{\langle x \rangle^n (u-v)}_{L^{\infty}}... \notag\\
        &\hspace{1.8in} ...\norm{\langle x \rangle^n u}^{\alpha+1}_{L^{\infty}}\underbrace{\Big\|\langle x \rangle ^{2kn}\prod_{j=1}^k \left(\partial_x^{m_{1,j}}u  \partial_x^{m_{2,j}}\bar{u}\right)\partial_x^\rho u\Big\|_{L^{2}}}_{(*)}.
\end{align}
The term $B_2$ is estimated by
\begin{equation}\label{unweighted_B2_part}
\norm{B_2}_{L^2}
            \leq C\eta^{2k} \norm{\langle x \rangle^{- \alpha n}}_{L^\infty} \norm{\langle x \rangle^n v}^{a}_{L_{\infty}} \underbrace{\Big\|\langle x\rangle^{2kn}\prod_{j=1}^{k}\left(\partial_{x}^{m_{1,j}}u\partial_{x}^{m_{2,j}}\bar{u}\right)\partial_{x}^{\rho}(u-v)\Big\|_{L^{2}}}_{(*)},
    \end{equation}
and $B_3$ is bounded by (on an example of $B_3^\prime)$ 
\begin{align}\label{unweighted_B3_part}
    \| B_3'\|_{L^2} 
    &\leq \eta^{2k}\| \langle x \rangle ^{-\alpha n}\|_{L^\infty} \|\langle x \rangle ^n v\|_{L^\infty}^{\alpha} \Bigg\|\langle x \rangle ^{2kn}\partial_x^\rho v \Bigg[\sum_{i=1}^k \left( \prod_{j=1}^{i-1} \partial^{m_{1,j}}v\partial^{m_{2,j}}\bar v \right) \notag\\
    &\hspace{4.3cm}\times
    (\partial^{m_{1,i}}u  - \partial^{m_{1,i}} v) ( \partial^{m_{2,i}}\bar u +\partial^{m_{2,i}}\bar v) \notag\\
    &\hspace{5.4cm}\times
    \left( \prod_{j=i+1}^{k}\partial^{m_{1,j}}u \partial^{m_{2,j}}\bar u \right)\Bigg] \Bigg\|_{L^2}.
\end{align}
With the estimates compiled, we notice that the last term in \eqref{unweighted_B3_part}, the $(*)$ term in \eqref{unweighted_I}, \eqref{unweighted_II}, and \eqref{unweighted_B2_part} all are of the same form as their counterparts in part {\bf III.} of the (single term) nonlinear estimates and the sub-cases considered there (recall, those are either: (1) all derivatives $m_{i,j}, \;\rho \leq r$ or (2) there is one derivative of order $r + 1$ or greater and the rest are smaller than $r$) are the same. Utilizing a similar argument in that case, we finalize those estimates as follows
\begin{equation}
            \|I \|_{L^2} \leq C\eta^{2+4k} \left(\norm{u}^{2k}_{\mathcal{X}} + \norm{v}^{2k}_{\mathcal{X}}\right) \norm{u-v}_{\mathcal{X}} \norm{u}^{2k+\alpha+2}_{\mathcal{X}}
\end{equation}
and
\begin{equation}
            \| II \|_{L^2}\leq C\eta^{1+2k}\left(\norm{u}_\mathcal{X}^\alpha + \norm{v}_\mathcal{X}^\alpha\right) \norm{u-v}_{\mathcal{X}}\norm{u}_\mathcal{X}^{2k+1}.
\end{equation}
Putting both terms together, we obtain the final estimate for $B_1$,
\begin{equation}\label{L2_pt2_B1_case2}
        \begin{aligned}
            \norm{ B_1}_{L^2} 
            &\leq C\eta^{1+2k}\norm{u-v}_{\mathcal{X}} \norm{u}_{\mathcal{X}}^{2k+1}\left(\norm{u}_{\mathcal{X}}^{\alpha} + \norm{v}_{\mathcal{X}}^{\alpha} +  \eta^{1+2k}\norm{u}_\mathcal{X}^{\alpha+1}\left[\norm{u}_{\mathcal{X}}^{2k} + \norm{v}_{\mathcal{X}}^{2k}\right]\right),
        \end{aligned}
    \end{equation}
which is analogous to \eqref{L2_pt1_B1}; this will be a trend as we continue on.

Indeed, we have the related estimate to \eqref{L2_pt1_B2} for $B_2$ in this case as
\begin{equation}
 \norm{B_2}_{L^2} \leq C\eta^{2k} \norm{u-v}_{\mathcal{X}} \norm{v}^{\alpha}_{\mathcal{X}} \norm{u}^{2k}_{\mathcal{X}}.\label{L2_pt2_B2_case2}
\end{equation}
For $B_3$ we have
\begin{equation}\label{L2_pt2_B3_case1}
    \|B_3\|_{L^2} \leq C\eta^{2k}(\|u\|_{\mathcal{X}}+\|v\|_{\mathcal{X}})^{2k+\alpha + 1}\|u-v\|_{\mathcal{X}}.
\end{equation}
Putting together \eqref{L2_pt_Aestimate}, \eqref{L2_pt2_B3_case1}, \eqref{L2_pt2_B2_case2}, \eqref{L2_pt2_B1_case2} we obtain the same bound as in \eqref{C2diffest}. Then,  combined with \eqref{diffLinfty}, we obtain \eqref{Thm3.1Pt2}, 
completing the proof of Theorem \ref{NonlinearEst}.
\end{proof}

\subsection{Proof of Theorems  \ref{T:3} and \ref{T:4}} 
We give the proof of Theorem \ref{T:4}; to show Theorem \ref{T:3} simply take  $\lambda_2, ..., \lambda_N = 0$. {In this method, we can close the fixed point argument for all nonlinear terms at once, whereas in Method I, we would have to deal with different admissible pairs for each nonlinear term and then find some mutual grounds, also for small nonlinearities $\alpha<1$ the first method is simply not applicable.} This shows the flexibility of the second method. 
\begin{proof}
We first rewrite the cNLS equation \eqref{NLS} with \eqref{cNLS} as an integral equation
\begin{equation}{\label{IntEq}}
    u(t) = e^{it\partial_{x}^{2}}u_{0} + i \int_{0}^{t} e^{i\partial_{x}^{2}(t-\tau)} \big( \sum_{i=1}^{N}\lambda_{i}(|u|^{\alpha_i}u) (\tau)\big) \, d\tau.
\end{equation}
We use a standard contraction mapping argument (as we did in the first method) using the linear estimates of Theorem \ref{linearEst} and the nonlinear estimates of Theorem \ref{NonlinearEst}. Let $\eta >0$. For $K>0$ and $T>0$ (specified later) we define the space $\mathcal{E}$ by 
\begin{align}\label{Espace}
\mathcal{E}_{K,T} = \big\{v\in C([-T,T], \mathcal{X}) : 
&\norm{v}_{L^\infty([-T,T],\mathcal{X})} \leq K \quad \mbox{and} \\ 
& \eta \, \langle x \rangle^n|v(x,t)| \geq \tfrac1{2} ~~  \mbox{for} ~~  (x,t)\in \mathbb{R}\times [-T,T]\big\}. \label{Espace2}
\end{align}
Note that $\mathcal{E}_{K,T}$, equipped with the metric $d(u,v) = \norm{u-v}_{L^{\infty}([-T,T],  \mathcal{X})}$, is a complete metric space.     Given $v\in \mathcal{E}_{K,T}$ and $v_0 \in     \mathcal{X}$,  for $-T < t < T$ using Duhamel's formula \eqref{IntEq}, we set
\begin{equation}\label{Phi}
\Phi (v(t)) = e^{it\partial_{x}^{2}} v_{0} + i\int_0^t    e^{i(t-\tau)\partial_x^2}
\big(\sum_{i=1}^{N} \lambda_i (|v|^{\alpha_i} v)(\tau) \big) \, d\tau =: e^{it\partial_{x}^{2}} v_{0} +\Psi_v(t).
\end{equation}

Applying \eqref{Prop1p2} to both the linear part and inside the integral and then \eqref{Thm3.1Pt1}, we obtain 
\begin{align*}
\norm{\Phi (v(t))}_{ \mathcal{X}} 
&\leq \| e^{it\partial_{x}^{2}} v_{0} \|_{\mathcal X} 
 + \sum_{i=1}^{N} |\lambda_i| \Big\| \int_0^t e^{i(t-\tau)\partial_x^2} (|v|^{\alpha_i} v)(\tau)\, d\tau \Big\|_\mathcal{X}  
\\
&\leq C \Big(\langle t \rangle^{n+1} \|v_0\|_{\mathcal X} + \sum_{i=1}^{N} |\lambda_i| \int_0^{|t|}(1+(|t|-\tau))^{n+1} \norm{(|v|^{\alpha_i}v)(\tau)}_\mathcal{X} \, d\tau  \Big)
\\
&\leq C \Big(\langle t \rangle^{n+1} \|v_0\|_{\mathcal X} + \sum_{i=1}^{N} 
|\lambda_ i| \int_0^{|t|}(1+(|t|-\tau))^{n+1} (1+\eta\norm{v(\tau)}_\mathcal{X})^{2M}\norm{v(\tau)}_\mathcal{X}^{\alpha_i + 1} \, d\tau \Big)
\end{align*}
Next, for $0< \tau \leq |t| \leq  T$, we estimate $(1+(|t|-\tau))^{n+1} \leq (1+|t|)^{n+1} \leq (1+T)^{n+1}$, and substituting it into the previous line, we get
\begin{align*}
\norm{\Phi (v(t))}_{ \mathcal{X}} 
&\leq C \Big( \langle t \rangle^{n+1} \|v_0\|_{\mathcal X} + \sum_{i=1}^{N} |\lambda_i| \, |t|(1+|t|)^{n+1}\sup_{\tau \in [0,|t|]} \big((1+\eta\norm{v(\tau)}_\mathcal{X})^{2M}\norm{v(\tau)}_\mathcal{X}^{\alpha_i + 1} \big) \Big)
\end{align*}
Recalling that $v \in \mathcal E_{K,T}$, and so $\|v (t) \|_\mathcal{X} \leq \|v\|_{L^\infty([-T,T],\mathcal{X})} \leq K$, we deduce  
\begin{equation}\label{E:Phi1}
\norm{\Phi(v)}_{L^\infty([-T,T], \mathcal{X})} 
\leq C \Big( \langle T\rangle^{n+1} \|v_0\|_{\mathcal X} + \sum_{i=1}^{N} |\lambda_i| (1+T)^{n+1}T \, 
(1+\eta K)^{2M}K^{\alpha_i + 1} \Big). 
\end{equation}

Now we can set the positive constants $K$ and $T$ to specify exactly what space $\mathcal E_{K,T}$ in \eqref{Espace} we use. 
First, note that if $0<T<1$, then $\langle T \rangle^{n+1} < 2^{n+1}$. Thus, by taking 
\begin{equation}\label{K}
K \defeq 2^{n+2} C \|v_0\|_{\mathcal{X}},
\end{equation}
we bound the first term in \eqref{E:Phi1} by $\frac{K}2$. 
Secondly, for the sum in \eqref{E:Phi1}, 
we note that for any $i=1, ..., N$, each term in the sum is bounded by $ 2^{n+1} C |\lambda_i| T (1+ \eta K)^{2M} K^{\alpha_i+1} =: \tilde C_i T K$, and thus, by choosing $T$ small enough, we can make this entire sum (so it is essential that we have {\it finitely many} terms) to be less than $\frac{K}2$. Thus, we choose $T_1 \in (0,1)$ such that
{\small 
\begin{align}\label{E:T_weight}
T_1< & \frac1{ C 2^{n+1} (1+ \eta K)^{2M} \sum_{i=1}^{N} |\lambda_i|  K^{\alpha_i}} \\
& \equiv 
\frac1{ C 2^{n+1} (1+ \eta 2^{n+2} C \|v_0\|_{\mathcal{X}})^{2M} \sum_{i=1}^{N} |\lambda_i|  (2^{n+2} C)^{\alpha_i} \|v_0\|_{\mathcal{X}}^{\alpha_i} }. \notag
\end{align}
}
Then substituting these bounds into \eqref{E:Phi1}, we deduce
\begin{equation}\label{E:Phi2}
\norm{\Phi(v)}_{L^\infty([-T_1,T_1], \mathcal{X})} 
\leq  \frac{K}2+\frac{K}2 = {K}.
\end{equation}

We next show the lower bound in \eqref{Espace2} for $\Phi(v(t))$, using the estimate \eqref{Prop1p2} and \eqref{E:Phi1}, to get
\begin{align}
\eta |\langle x \rangle^n \Phi(v)(x,t)| 
& \geq \eta  
|\langle x \rangle^n v_0(x)| 
- \eta \|\langle x \rangle^n ( e^{it\partial_x^2} v_0-v_0) \|_{L^{\infty}} - \eta \Big\|\int_0^t e^{i(t-\tau)\partial_x^2} \big(
\sum_{i=1}^{N} \lambda_i (|v|^{\alpha_i} v) \big) (\tau) 
\, d\tau \bigg\|_{\mathcal{X}} \notag\\
& \geq 1 - \eta \, C T\langle T \rangle^n\|v_0\|_{\mathcal{X}}
- \eta \, C T\langle T \rangle^{n+1} (1+\eta K)^{2M}
\sum_{i=1}^{N} |\lambda_i| \, 
K^{\alpha_i + 1}, \notag\\
& \geq 1 - \eta \, C T  \Big(\langle T \rangle^n\|v_0\|_{\mathcal{X}}
- \langle T \rangle^{n+1} (1+\eta K)^{2M}
\sum_{i=1}^{N} |\lambda_i| \, 
K^{\alpha_i + 1} \Big),\label{E:eta3}
\end{align}
and re-writing the expression in parentheses in the second term in \eqref{E:eta3}, we can bound it for $0<T<1$ (and thus, $\la T \ra^n<2^n$) as 
\begin{align*}
& \langle T \rangle^n
\|v_0 \|_{\mathcal X} + \langle T \rangle^{n+1} (1+\eta K)^{2M}
\sum_{i=1}^{N} |\lambda_i| \, 
K^{\alpha_i + 1} <  K (1+ \eta C 2^{n+1} (1+\eta K)^{2M} \sum_{i=1}^{N} |\lambda_i| \, 
K^{\alpha_i}) =: \tilde C_{\eta, n, v_0}.
\end{align*}
Hence, 
$$
\eta | \la x \ra^n \Phi(v)(x,t)| 
\geq 1 - \eta C \, \tilde C_{\eta, n, v_0} \, T \geq \frac12,
$$
provided that we choose  $T<T_2$, where
\begin{equation}\label{E:T2}
T_2 < \frac1{2 \,\eta  C \, \tilde C_{\eta, n, v_0}}.
\end{equation}
Take $\tilde T = \min \{T_1, T_2\}$, the we can 
conclude that $\Phi$ maps $\mathcal E_{K,\tilde T}$ into itself for this $\tilde T$ and $K$ from \eqref{K}. 

We next show that $\Phi$ is a contraction on $\mathcal E_{K,T}$ for the given $K$ and time $T < \tilde T$ (specified below). Taking $u(t)$ and $v(t)$ solutions stemming from the initial condition $v_0$, and for $|t| < T$ (to consider nonlinear evolution in both time directions), using similar steps as we did in obtaining \eqref{E:Phi1}, we have
\begin{align}
&\norm{\Phi(u(t)) - \Phi (v(t))}_{\mathcal{X}}  = \norm{\Psi_u(t) - \Psi_v(t)}_{\mathcal{X}} 
= \Big\| \int_0^{|t|} e^{i(t-\tau)\partial_x^2} \sum_{i=1}^N \lambda_i \left(|u|^{\alpha_i}u - |v|^{\alpha_i}v \right)(\tau) \, d\tau \Big\|_\mathcal{X} \\
    & \qquad \qquad \leq C \sum_{i=1}^{N} |\lambda_i|  (1+T)^{n+1} \int_0^{|t|} \norm{\left(|u|^{\alpha_i}u - |v|^{\alpha_i}v \right)(\tau)}_\mathcal{X}\, d\tau \\
&\qquad \leq C T(1+T)^{n+1} \sum_{i=1}^{N} |\lambda_i| (1+\eta)^{4M +\alpha_i+ 2} \left(1 + \norm{u(t)}_\mathcal{X} + \norm{v(t)}_\mathcal{X}\right)^{4M + \alpha_i + 2}\norm{u(t)-v(t)}_\mathcal{X},
\end{align}
where in the last step we used \eqref{Thm3.1Pt2}. 
Then, taking sup in time and applying the fact that $u,v \in \mathcal E_{K,T}$ and thus, $\|u(t)\|_{\mathcal X} \leq K$ for any $t \in [-T,T]$, we have
\begin{align}\label{E:Psi-diff}
\norm{\Psi_u - \Psi_v}_{L^\infty([-T,T],\mathcal{X})} 
\leq C T (1+T)^{n+1} 
\Big( \sum_{i=1}^{N} |\lambda_i| (1+2K)^{4M+\alpha_i+2} \Big)
\|u - v \|_{L^\infty([-T,T],\mathcal X)}.
\end{align}
Choose $0<T_3<1$ such that 
\begin{equation}\label{E:T3}
T_3 \leq \frac1{C \,  2^{n+2} 
\big( \sum_{i=1}^{N} |\lambda_i| (1+2K)^{4M+\alpha_i+2} \big)}.
\end{equation}
Let $T = \min \{ \tilde T, T_3\} \equiv \min \{T_1, T_2, T_3\}$. Then 
$$
\norm{\Psi_u - \Psi_v}_{L^\infty([-T,T],\mathcal{X})} 
\leq \frac12 \|u - v \|_{L^\infty([-T,T],\mathcal X)},
$$
implying that $\Phi$ is a contraction on $\mathcal E_{K,T}$, which concludes existence and uniqueness of solutions in $\mathcal X$ to the cNLS \eqref{NLS} with \eqref{cNLS}.

Continuous dependence follows as a consequence of the difference estimates, by taking  $u_0, v_0 \in X$ such that $\|u_0\|_{\mathcal X}, \|v_0\|_{\mathcal X} < K$ and both $u_0$ and $v_0$ satisfy the vanishing condition as in \eqref{Espace2}. Then considering two solutions $u(t)$ and $v(t)$ written via the integral equation \eqref{IntEq}, which start from $u_0\in X$ and $v_0 \in X$, respectively, and estimating the difference $\norm{\Phi (u) - \Phi(v)}_{L^\infty((-T,T),\mathcal{X})}$ analogously as above, one concludes a similar expression as in \eqref{E:CD} via a similar argument as outlined in Section \ref{S:contdep}.
\end{proof}

\begin{remark}
Both methods can be applied to other nonlinear equations: the first method is well-known, its further applications could be found in \cite{Caz-book, LPIntroToNDE2014, tao2006nonlinear}. The second is more recent and has been applied to a few different models, including equations with different dispersion (for example, cubic as in KdV equation \cite{LMG},\cite{Diana}, or quartic as in the bi-harmonic NLS equation, \cite{Iryna}).
\end{remark}

\bibliographystyle{abbrv}
{\footnotesize
\bibliography{main.bib}
}
\end{document}